\documentclass{article}

\usepackage{amsmath, amsthm, amsfonts}
\usepackage{a4wide}
\usepackage{amsmath, array}
\usepackage{url}
\usepackage{enumitem}

\newtheorem{thm}{Theorem}[section]
\newtheorem{cor}[thm]{Corollary}

\newtheorem{lem}[thm]{Lemma}
\newtheorem{prop}[thm]{Proposition}
\theoremstyle{definition}

\newtheorem{rem}[thm]{Remark}

\def\ZZ{\mathbb{Z}}
\def\NN{\mathbb{N}}
\def\CC{\mathbb{C}}

\newcommand{\gra}{{\alpha}} \newcommand{\grb}{{\beta}} \newcommand{\grg}{{\gamma}} 
 \newcommand{\grz}{{\zeta}}  \newcommand{\gru}{{\theta}}
  \newcommand{\grl}{{\lambda}} \newcommand{\grm}{{\mu}}
   \newcommand{\grp}{{\pi}}
 \newcommand{\grs}{{\sigma}}  
\newcommand{\grf}{{\phi}} \newcommand{\grx}{{\chi}} \newcommand{\grc}{{\psi}} \newcommand{\grv}{{\omega}}

\newcommand{\grF}{{\Phi}}   
\newcommand{\arw}{\rightarrow} 

\title{Universal deformations of the finite quotients of the braid group on 3 strands}
\author{Eirini Chavli}

\begin{document}
\maketitle

\noindent
{\bf Abstract.} We prove that  the quotients of the group algebra of the braid group on 3 strands by a generic  quartic and quintic relation respectively, have finite rank. This is a special case of a conjecture by Brou\'{e}, Malle and Rouquier for the generic Hecke algebra of an arbitrary complex reflection group. Exploring the consequences of this case, we prove that we can determine completely the irreducible representations of this braid group of dimension at most 5, thus recovering a classification of Tuba and Wenzl in a more general framework.\\ \\
\textbf{Acknowledgments.} I would like to thank my supervisor Ivan Marin for his help and support during this research, and Maria Chlouveraki, for fruitful discussions on the last section of this paper.  I would also like to thank the University Paris Diderot - Paris 7 for its financial support.

\section{Introduction}
\indent

In 1999 I. Tuba and H. Wenzl classified the irreducible representations of the braid group $B_3$ of dimension $k$ at most 5 over an algebraically closed field $K$ of any characteristic (see \cite{T}) and, therefore, of $PSL_2(\ZZ)$, since the quotient group $B_3$ modulo its center is isomorphic to $PSL_2(\ZZ)$. Recalling that $B_3$ is given by generators $s_1$ and $s_2$ that satisfy the relation $s_1s_2s_1=s_2s_1s_2$, we assume that $s_1\mapsto A, s_2\mapsto B$ is an irreducible representation of $B_3$, where $A$ and $B$ are invertible $k\times k$ matrices over $K$ satisfying $ABA=BAB$. I. Tuba and H. Wenzl proved that $A$ and $B$ can be chosen to be in \emph{ordered triangular form}\footnote{Two $k\times k$ matrices are in ordered triangular form if one of them is an upper triangular matrix with eigenvalue $\grl_i$ as $i$-th diagonal entry, and the other is a lower triangular matrix with eigenvalue $\grl_{k+1-i}$
as $i$ -th diagonal entry.} with coefficients completely determined by the eigenvalues (for $k\leq3$) or by the eigenvalues and by the choice of a $k$th root of det$A$ (for $k>3$). Moreover, they proved that such irreducible representations exist if and only if the eigenvalues do not annihilate some polynomials $P_k$ in the eigenvalues and the choice of the $k$th root of det$A$, which they determined explicitly. 

At this point, a number of questions arise: what is the reason we do not expect their methods to work for any dimension beyond 5 (see \cite{T}, remark 2.11, 3)? Why are the matrices in this neat form? In \cite{T}, remark 2.11, 4 there is an explanation for the nature of the polynomials $P_k$. However, there is no argument connected with the nature of $P_k$ that explains the reason why these polynomials provide  a  necessary condition for a representation of this form to be irreducible. In this paper we answer these questions by recovering this classification of the irreducible representations of the braid group $B_3$ as a consequence of the freeness conjecture for the generic Hecke algebra of the finite quotients of the braid group $B_3$, defined by the additional relation $s_i^k=1$, for $i=1,2$ and $2\leq k\leq 5$. For this purpose, we first prove this conjecture for $k=4,5$ (the rest of the cases are known by previous work). The fact that there is a connexion between the classification of the irreducible representations of $B_3$ of dimension at most 5 and its finite quotients has already been suspected by I. Tuba and H. Wenzl (see \cite{T}, remark 2.11, 5).

More precisely, there is a Coxeter's classification of the finite quotients of the braid group $B_n$ on $n$ strands  by the additional relation $s_i^k=1$, for $i=1,2$ (see \cite{Coxeter}); these quotients are finite if and only if  $\frac{1}{k}+\frac{1}{n}>\frac{1}{2}$. If we exclude the obvious cases $n= 2$ and $k=2$, which lead to the
cyclic groups and to the symmetric groups respectively, there is only a finite number of such groups, which
are irreducible complex reflection groups: these are the groups $G_4, G_8$ and $G_{16}$, for $n=3$ and $k=3,4,5$ and the groups  $G_{25}, G_{32}$ for $n=4,5$ and $k=3$, as they are known in the Shephard-Todd classification (see \cite{S}). Therefore, if we restrict ourselves to the case of $B_3$, we have the finite quotients $W_k$, for $2\leq k\leq 5$, which are the groups $\mathfrak{S}_3, G_4, G_8$ and $G_{16}$, respectively.

We set $R_k=\ZZ[a_{k-1},...,a_1, a_0, a_0^{-1}]$, for $k=2,3,4,5$ and we denote by $H_k$  the \emph{generic Hecke algebra} of $W_k$; that is the quotient of the group algebra $R_kB_3$ by the relations $s_i^k=a_{k-1}s_i^{k-1}+...+a_1s_i+a_0$, for $i=1,2$. We assume we have an irreducible representation of $B_3$ of dimension $k$ at most 5. By the Cayley-Hamilton theorem of linear algebra, the image of a generator under such a representation is annihilated by a monic polynomial $m(X)$ of degree $k$, therefore this representation has to factorize through the corresponding Hecke algebra  $H_k$. As a result, if $\gru: R_k\rightarrow K$ is a specialization of $H_k$ such that $a_i\mapsto m_i$, where $m_i$ are the coefficients of $m(X)$, the irreducible representations of $B_3$ of dimension $k$ are exactly the irreducible representations of $H_k\otimes_{\gru}K$ of dimension $k$. A conjecture of  Brou\'{e}, Malle and Rouquier states that $H_k$ is free as $R_k$-module of rank $|W_k|$. Based on this assumption, the irreducible representations of $H_k$ have been determined in \cite{Mallem}. We will show how to use the decomposition map $d_{\gru}$ (see \cite{Geck} \textsection 7.3), in order to get the irreducible representations of $H_k\otimes_{\gru}K$ that we are interested in.

The general freeness conjecture of Brou\'{e}, Malle and Rouquier states that the generic Hecke algebra of a complex reflection group is a free $R$-module of finite rank, where $R$ is the ring of definition of the Hecke algebra (see \cite{BMR}). For the finite quotients $W_k$ of the braid group we mentioned before, this conjecture is known to be true for the symmetric group (see \cite{Geck}, Lemma 4.4.3), and it was proved in \cite{F}, \cite{BM} and \cite{Ivan} for the case of $G_4$ and in \cite{Ivan} for the cases of $G_{25}$ and $G_{32}$. We will  prove the validity of the conjecture  for the rest of the cases, which belong to the class of complex reflection groups of rank two\footnote{The study of the conjecture of these groups is the subject of the author's PhD thesis, that is still in progress.}; the main theorem of this paper is the following:
\begin{thm}$H_k$ is a free $R_k$-module of rank $|W_k|$.
\end{thm}
By general arguments (see, for example, \cite{ivan}) this has for consequence the following:
\begin{cor} If $F$ is a suitably large extension of the field of fractions of $R_k$, then $H_k\otimes_{R_k}F$ is isomorphic to the group algebra $FW_k$.
\end{cor}
In order to prove this theorem we need some preliminary results, which contain a lot of calculations between the images of some elements of the braid group inside the Hecke algebra. We hope that this will not discourage the reader to study the proof, since these calculations are not that complicated and they should be fairly easy to follow.

\section{Preliminaries}
\label{s}
\indent

Let $B_3$ be the braid group on 3 strands, given by generators the braids $s_1$ and $s_2$ and the single relation $s_1s_2s_1=s_2s_1s_2$, that we call braid relation.

We set $R_k=\ZZ[a_{k-1},...,a_1, a_0, a_0^{-1}]$, for $k=2,3,4,5$. Let $H_k$ denote the quotient of the group algebra $R_kB_3$ by the relations \begin{equation}s_i^k=a_{k-1}s_i^{k-1}+...+a_1s_i+a_0,\label{one} \end{equation}
for $i=1,2$. For $k=2, 3,4$ and 5 we call $H_k$ the quadratic, cubic, quartic and quintic Hecke algebra, respectively. 

We identify $s_i$ to their images in $H_k$. We multiply $(\ref{one})$ by $s_i^{-k}$ and since $a_0$ is invertible in $R_k$ we have:
\begin{equation}s_i^{-k}=-a_0^{-1}a_1s_i^{-k+1}-a_0^{-1}a_2s_i^{-k+2}-...-a_0^{-1}a_{k-1}s_i^{-1}+a_0^{-1}, \label{two}\end{equation}
for $i=1,2$. 
If we multiply ($\ref{two}$) with a suitable power of $s_i$ we can expand $s_i^{-n}$ as a linear combination of 
$s_i^{-n+1}$,$\dots$, $s_i^{-n+(k-1)}$, $s_i^{-n+k}$, for every $n\in \NN$.
Moreover, comparing (\ref{one}) and (\ref{two}), we can define an automorphism $\grF$  of $H_k$ as $\ZZ$-algebra, where $$\begin{array}{lcl}
s_i &\mapsto &s_i^{-1}, \text { for } i=1,2\\
a _j&\mapsto&-a_0^{-1}a_{k-j},\text{ for } j=1,...,k-1\\
a_0&\mapsto&a_0^{-1}
\end{array}$$

We will prove now an easy lemma that plays an important role in the sequel. This lemma is in fact a generalization of  lemma 2.1 of \cite{Ivan}. 
\begin{lem}For every $m \in \ZZ$ we have  
	$s_2s_1^ms_2^{-1}=s_1^{-1}s_2^ms_1$ and $s_2^{-1}s_1^ms_2=s_1s_2^ms_1^{-1}$.
	\label{lem1}
\end{lem}
\begin{proof}
	By using the braid relation we have that $(s_1s_2)s_1(s_1s_2)^{-1}=s_2$. Therefore, for every $m \in \ZZ$ we have $(s_1s_2)s_1^m(s_1s_2)^{-1}=s_2^m$, that gives us the first equality. Similarly, we prove the second one.
\end{proof}
If we assume $m$ of lemma \ref{lem1} to be positive we have  $s_1s_2s_1^n=s_2^ns_1s_2$ and $s_1^ns_2s_1=s_2s_1s_2^n$, where $n\in \NN$. Taking inverses, we also get $ s_1^{-n}s_2^{-1}s_1^{-1}=s_2^{-1}s_1^{-1}s_2^{-n}$ and $ s_1^{-1}s_2^{-1}s_1^{-n}=s_1^{-n}s_2^{-1}s_1^{-1}$. We call all the above relations \emph{the generalized braid relations.}

We denote by $u_i$ the $R_k$-subalgebra of $H_k$ generated by $s_i$ (or equivalently by $s_i^{-1}$) and by $u_i^{\times}$ the group of units of $u_i$, where $i=1,2$. We also set $\grv:=s_2s_1^2s_2$. Since the center of $B_3$ is the subgroup generated by the element $z=s_1^2\grv$ (see, for example, theorem 1.24 of \cite{Turaev}), for all $x\in u_1$ and $m\in \ZZ$ we have that $x\grv^m=\grv^mx$. We will see later that $\grv$ plays an important role in the description of $H_k$.

Let $W_k$ be the quotient group $B_3/\langle s_i^k \rangle$, 
$k=2, 3, 4$ and 5. From a Coxeter's theorem (see \textsection10 in \cite{Coxeter}) we know that $W_k$ is finite. Let $r_k$ denote the order of $W_k$.
Our goal now is to prove that $H_k$ is a free $R_k$-module of rank $r_k$, a statement that holds for $H_2$ since $W_2=\mathfrak{S}_3$ is a Coxeter group (see \cite{Geck}, Lemma 4.4.3).
For the remaining cases, we will use the following proposition.
\begin{prop}Let $k\in\{3,4,5\}$. If $H_k$ is generated as a module over $R_k$ by $r_k$ elements, then $H_k$ is a free $R_k$-module of rank $r_k$.
	\label{rp}
\end{prop}
\begin{proof}The algebras $H_k$ are the generic Hecke algebras of the complex reflection groups $G_4, G_8$ and $G_{16}$, respectively in the sense of Brou\'e, Malle and Rouquier (see \cite{BMR}). Hence, the result follows from theorem 4.24 in \cite{BMR} or from proposition 2.4(1) in \cite{ivan}.
\end{proof}
Therefore, we need to find a spanning set of $H_k$, $k=3,4,5$  of $r_k$ elements. However, for the cubic Hecke algebra $H_3$ we have that $H_3=u_1+u_1s_2u_1+u_1s_2^{-1}u_1+u_1s_2s_1^{-1}s_2$ (see \cite{Ivan}, theorem 3.2(3)) and, hence, $H_3$ is spanned as a left $u_1$-module by 8 elements. Since $u_1$ is spanned by 3 elements as $R_3$-module, we have that $H_3$ is spanned over $R_3$ by $r_3=24$ elements.

\section{The quartic Hecke algebra $H_4$}
\label{sec2}
\indent

Our ring of definition is $R_4=\ZZ[a, b,c,d, d^{-1}]$ and therefore, relation (\ref{one}) becomes  $s_i^4=as_i^3+bs_i^2+cs_i+d$, for $i=1,2$. We set

 $$\begin{array}{lcl}
U'&=&u_1u_2u_1+u_1s_2s_1^{-1}s_2u_1+
u_1s_2^{-1}s_1s_2^{-1}u_1+u_1s_2^{-1}s_1^{-2}s_2^{-1} \\  U&=&U'+u_1s_2s_1^{-2}s_2u_1+u_1s_2^{-2}s_1^{-2}s_2^{-2}u_1.\end{array}$$
It is obvious that $U$ is a $u_1$-bimodule and that $U'$ is a $u_1$-sub-bimodule of $U$. Before proving our main theorem (theorem \ref{th}) we  need a few preliminaries results. 
\begin{lem}For every $m \in \ZZ$ we have 
	\begin{itemize}[leftmargin=0.6cm]
		\item[(i)] $s_2s_1^{m}s_2\in U$.
		\item [(ii)]$s_2^{-1}s_1^{m}s_2^{-1}\in U'$.
		\item [(iii)]$s_2^{-2}s_1^{m}s_2^{-1}\in U'$.
	\end{itemize}
	\label{lem2}
\end{lem}
\begin{proof}By using the relations (\ref{one}) and (\ref{two}) we can assume that $m\in\{0,1,-1,-2\}$. Hence, we only have to prove (iii), since (i) and (ii)  follow from the definition of $U$ and $U'$ and the braid relation.  For (iii), we can assume that $m\in\{-2, 1\}$, since the case where $m=-1$  is obvious by using the generalized braid relations. We have: 
	$s_2^{-2}s_1^{-2}s_2^{-1}=s_1^{-1}(s_1s_2^{-2}s_1^{-1})s_1^{-1}s_2^{-1}
	=s_1^{-1}s_2^{-1}s_1^{-2}(s_2s_1^{-1}s_2^{-1})
	=s_1^{-1}(s_2^{-1}s_1^{-3}s_2^{-1})s_1.$ The result then follows from $(ii)$.
	For the element $s_2^{-2}s_1s_2^{-1}$, we expand $s_2^{-2}$ as  a linear combination of $s_2^{-1}, 1, s_2, s_2^2$ and by using the definition of $U'$ and lemma $\ref{lem1}$,  we only have to check that $s_2^2s_1s_2^{-1}\in U'$. Indeed, we have: $s_2^2s_1s_2^{-1}= s_2(s_2s_1s_2^{-1})=(s_2s_1^{-1}s_2)s_1
	\in U'.$
\end{proof}
\begin{prop}$u_2u_1u_2\subset U.$
	\label{prop1}
\end{prop}
\begin{proof}We need to prove that every element of the form $s_2^{\gra}s_1^{\grb}s_2^{\grg}$ belongs to $U$, for $\gra,\grb,\grg\in\{-2,-1,0,1\}$. However, when $\gra\grb\grg=0$ the result is obvious. Therefore, we can assume $\gra,\grb,\grg\in\{-2,-1,1\}$. We have the following cases:
	\begin{itemize}[leftmargin=*]
		\item\underline{$\gra=1$}:
		The cases where $\grg\in\{-1, 1\}$ follow from lemmas $\ref{lem1}$ and \ref{lem2}(i). Hence, we need to prove that $s_2s_1^{\grb}s_2^{-2}\in U$. 
		For $\grb=-1$ we use lemma $\ref{lem1}$ and we have $s_2s_1^{-1}s_2^{-2}=(s_2s_1^{-1}s_2^{-1})s_2^{-1}=s_1^{-1}(s_2^{-1}s_1s_2^{-1})\in U.$
		For $\grb=1$ we  expand $s_2^{-2}$ as  a linear combination of $s_2^{-1}, 1, s_2, s_2^2$ and the result follows from the cases where $\grg\in\{-1,0,1\}$ and the generalized braid relations. It remains 
		to prove that $s_2s_1^{-2}s_2^{-2}\in U$. By expanding now $s_1^{-2}$ as  a linear combination of $s_1^{-1}, 1, s_1, s_1^2$  we only need to prove that $s_2s_1^2s_2^{-2}\in U$ (the rest of the cases correspond to $b=-1$, $b=0$ and $b=1$).
	We use  lemma $\ref{lem1}$ and we have:
		$s_2s_1^2s_2^{-2}=(s_2s_1^2s_2^{-1})s_2^{-1}=s_1^{-1}s_2(s_2s_1s_2^{-1})=
		s_1^{-1}(s_2s_1^{-1}s_2)s_1 \in U.$
		\item \underline{$\gra=-1$}: Exactly as in the case where $\gra=1$, we only have to prove that $s_2^{-1}s_1^{\grb}s_2^{-2}\in U$. For $\grb=-1$ the result is obvious by using the generalized braid relations. For $\grb=-2$ we have:
		$s_2^{-1}s_1^{-2}s_2^{-2}=(s_2^{-1}s_1^{-2}s_2)s_2^{-3}=s_1s_2^{-1}(s_2^{-1}s_
		1^{-1}s_2^{-3})
		=s_1(s_2^{-1}s_1^{-3}s_2^{-1})s_1^{-1}$. However,  by lemma \ref{lem2}(ii) we have that the element $s_2^{-1}s_1^{-3}s_2^{-1}$ is inside $U'$ and, hence, inside $U$. It remains to prove that
		$s_2^{-1}s_1s_2^{-2}\in U$.  For this purpose, we expand $s_2^{-2}$ as  a linear combination of $s_2^{-1}, 1, s_2, s_2^2$ and by the definition of $U$ and lemma \ref{lem1} we only need to prove that $s_2^{-1}s_1s_2^{2}\in U$.
		Indeed, using lemma \ref{lem1} again we have:
		$s_2^{-1}s_1s_2^2=(s_2^{-1}s_1s_2)s_2
		=s_1(s_2s_1^{-1}s_2)
		\in U.$  \item
		\underline{$\gra=-2$}:
		We can assume that $\grg\in\{1, -2\}$, since the case where $\grg=-1$ follows immediately from lemma \ref{lem2}$(iii)$.
		For $\grg=1$ we use lemma $\ref{lem1}$ and we have $s_2^{-2}s_1^{\grb}s_2=s_2^{-1}(s_2^{-1}s_1^{\grb}s_2)=(s_2^{-1}s_1s_2^{\grb})s_1^{-1}$. The latter is an element in $U$, as we proved in the case where $\gra=-1$. For $\grg=-2$ we only need to prove the cases where $\grb=\{-1,1\}$, since the case where $\grb=-2$ follows from the definition of $U$. We use the generalized braid relations and we have  $s_2^{-2}s_1^{-1}s_2^{-2}=(s_2^{-2}s_1^{-1}s_2^{-1})s_2^{-1}=s_1^{-1}(s_2^{-1}s_1^{-2}s_2^{-1})\in U$. Moreover,  $s_2^{-2}s_1s_2^{-2}=s_1(s_1^{-1}s_2^{-2}s_1)s_2^{-2}= s_1(s_2s_1^{-2}s_2^{-3})$. The result follows from the case where $\gra=1$, if we expand $s_2^{-3}$ as a linear combination of $s_2^{-2}$, $s_2^{-1}$, 1 and $s_2$.
			\qedhere
	\end{itemize}
\end{proof}
We can now prove the main theorem of this section.
\begin{thm}
	\mbox{}  
	\vspace*{-\parsep}
	\vspace*{-\baselineskip}\\ 
	\begin{itemize}[leftmargin=0.6cm]
		\item [(i)]
		$U=u_1u_2u_1+u_1s_2s_1^{-1}s_2u_1
		+u_1s_2^{-1}s_1s_2^{-1}u_1
		+u_1\grv
		+u_1\grv^{-1}
		+u_1\grv^{-2}$.
		\item[(ii)]$H_4=U$.
	\end{itemize}
	\label{th}
\end{thm}
\begin{proof}\mbox{}  
	\vspace*{-\parsep}
	\vspace*{-\baselineskip}\\ 
	\begin{itemize}[leftmargin=0.6cm]
		\item [(i)]
		We recall that $\grv=s_2s_1^2s_2$.
		We must prove that the RHS, which is by definition $U'+u_1\grv
		+u_1\grv^{-2}$, is  equal to $U$. For this purpose we will ``replace''  inside the definition of $U$ the elements $s_2s_1^{-2}s_2$ and $s_2^{-2}s_1^{-2}s_2^{-2}$ with the elements $\grv$ and $\grv^{-2}$ modulo $U'$, by proving that $s_2s_1^{-2}s_2\in u_1^{\times}\grv +U'$ and $s_2^{-2}s_1^{-2}s_2^{-2}\in u_1^{\times}\grv^{-2} +U'$. 
		
		For the element $s_2s_1^{-2}s_2$, we expand $s_1^{-2}$ as a linear combination of $s_1^{-1}, 1, s_1, s_1^2$, where the coefficient of $s_1^2$ is invertible. The result then follows from  the definition of $U'$ and the braid relation.
		For the element $s_2^{-2}s_1^{-2}s_2^{-2}$ we apply lemma \ref{lem1} and the generalized braid relations and we have: 
		$s_2^{-2}s_1^{-2}s_2^{-2}=s_2^{-2}s_1^{-1}(s_1^{-1}s_2^{-2}s_1)s_1^{-1}=s_2^{-1}(s_2^{-1}s_1^{-1}s_2)s_1^{-1}(s_1^{-1}s_2^{-1}s_1)s_1^{-2}=s_2^{-1}s_1(s_2^{-1}s_1^{-2}s_2)s_1^{-1}s_2^{-1}s_1^{-2}\in s_2^{-1}s_1^2s_2^{-2}s_1^{-2}s_2^{-1}u_1$.
		We expand $s_1^2$ as a linear combination of $s_1$, 1, $s_1^{-1}$, $s_1^{-2}$, where the coefficient of $s_1^{-2}$ is invertible and by the generalized braid relations and the fact that $s_1^{-2}\grv^{-2}=\grv^{-2}s_1^{-2}=s_2^{-1}s_1^{-2}s_2^{-2}s_1^{-2}s_2^{-1}s_1^{-2}$ we have that $$s_2^{-2}s_1^{-2}s_2^{-2}
		\in s_2^{-1}s_1s_2^{-2}s_1^{-2}s_2^{-1}u_1+s_2^{-3}s_1^{-2}s_2^{-1}u_1+u_1s_2^{-1}s_1^{-3}s_2^{-1}u_1+
		u_1^{\times}\grv^{-2}.$$
		Therefore, by lemma \ref{lem2}(ii) it is enough to prove that the elements $s_2^{-1}s_1s_2^{-2}s_1^{-2}s_2^{-1}$ and $s_2^{-3}s_1^{-2}s_2^{-1}$ belong to $U'$. However, the latter is an element in $U'$,  if we expand $s_2^{-3}$ as a linear combination of $s_2^{-2}, s_2^{-1}, 1, s_2$ and use lemma \ref{lem2}(iii), the definition of $U'$ and lemma \ref{lem1}. Moreover,
		$s_2^{-1}s_1s_2^{-2}s_1^{-2}s_2^{-1}=s_2^{-2}(s_2s_1s_2^{-1})\grv^{-1}=
		s_2^{-2}s_1^{-1}s_2s_1\grv^{-1}=s_2^{-2}s_1^{-1}s_2\grv^{-1}s_1^{-1}=(s_2^{-2}s_1^{-4}s_2^{-1})s_1^{-1}
		\in U'$, by lemma \ref{lem2}(iii).
		
		\item [(ii)]
		Since $1\in U$, it will be sufficient to show that $U$ is a left ideal of $H_4$. We know that $U$ is a $u_1$-sub-bimodule of $H_4$. Therefore, we only need to prove that $s_2U\subset U$. Since $U$ is equal to the RHS of (i) we have that
		$$s_2U\subset s_2u_1u_2u_1+s_2u_1s_2s_1^{-1}s_2u_1+s_2u_1s_2^{-1}s_1s_2^{-1}u_1+ s_2u_1\grv+s_2u_1\grv^{-1}+s_2u_1\grv^{-2}.$$ However, $s_2u_1u_2u_1+s_2u_1\grv+s_2u_1\grv^{-1}+s_2u_1\grv^{-2}=s_2u_1u_2u_1+s_2\grv u_1+s_2\grv^{-1}u_1+s_2\grv^{-2}u_1=s_2u_1u_2u_1+s_2^3s_1^2s_2u_1+s_1^{-2}s_2^{-1}u_1+s_1^{-2}s_2^{-2}s_1^{-1}s_2^{-1}
		\subset u_1u_2u_1u_2u_1$. Furthermore,
		by using lemma \ref{lem1} we have that $s_2u_1s_2^{-1}=s_1^{-1}u_2s_1$. Hence, $
		s_2u_1s_2s_1^{-1}s_2u_1=(s_2u_1s_2^{-1})s_2^2s_1^{-1}s_2u_1=s_1^{-1}u_2(s_1s_2^2s_1^{-1})s_2u_1=
		s_1^{-1}u_2s_1^2s_2^2u_1\subset u_1u_2u_1u_2u_1.$ Moreover, by using \ref{lem1} again we have that $(s_2u_1s_2^{-1})s_1s_2^{-1}u_1=s_1^{-1}u_2s_1^2s_2^{-1}u_1\subset u_1u_2u_1u_2u_1$ . Therefore, $$ s_2u_1u_2u_1+s_2u_1s_2s_1^{-1}s_2u_1+s_2u_1s_2^{-1}s_1s_2^{-1}u_1+ s_2u_1\grv+s_2u_1\grv^{-1}+s_2u_1\grv^{-2}\subset u_1u_2u_1u_2u_1.$$
		The result follows directly from proposition \ref{prop1}.
		\qedhere
	\end{itemize}
\end{proof}
\begin{cor}$H_4$ is a free $R_4$-module of rank $r_4=96$.
	\label{G8}
\end{cor}
\begin{proof}By proposition \ref{rp} it will be sufficient to show that $H_4$ is generated as $R_4$-module by $r_4=96$ elements. By theorem  \ref{th} and the fact that $u_1u_2u_1=u_1(R_4+R_4s_2+R_4s_2^{-1}+R_4s_2^2)u_1=u_1+u_1s_2u_1+u_1s_2^{-1}u_1+u_1s_2^2u_1$ we have that $H_4$ 
	is generated as left $u_1$-module by 24 elements. Since $u_1$ is generated by 4 elements as a $R_4$-module, we have that $H_4$ is generated over $R_4$ by 96 elements.
\end{proof}

\section{The quintic Hecke algebra $H_5$}
\indent

Our ring of definition is $R_5=\ZZ[a, b,c, d,e,e^{-1}]$ and therefore, relation (\ref{one}) becomes $s_i^5=as_i^4+bs_i^3+cs_i^2+ds_i+e$, for $i=1,2$.
We recall that $\grv=s_2s_1^2s_2$ and we set 
$$\small{\begin{array}{lcl}U'&=&u_1u_2u_1+u_1\grv+u_1\grv^{-1}+
u_1s_2^{-1}s_1^2s_2^{-1}u_1+u_1s_2s_1^{-2}s_2u_1+u_1s_2^2s_1^2s_2^{2}u_1+
u_1s_2^{-2}s_1^{-2}s_2^{-2}u_1+\\&&+u_1s_2s_1^{-2}s_2^{2}u_1+
u_1s_2^{-1}s_1^2s_2^{-2}u_1+u_1s_2^{-1}s_1s_2^{-1}u_1
+u_1s_2s_1^{-1}s_2u_1+u_1s_2^{-2}s_1^{-2}s_2^{2}u_1+
u_1s_2^{2}s_1^2s_2^{-2}u_1+\\&&+u_1s_2^{2}s_1^{-2}s_2^{2}u_1+
u_1s_2^{-2}s_1^2s_2^{-2}u_1+u_1s_2^{-2}s_1s_2^{-1}u_1+u_1s_2^{-1}s_1s_2^{-2}u_1\\ \\
U''&=&U'+u_1\grv^2+u_1\grv^{-2}+
u_1s_2^{-2}s_1^2s_2^{-1}s_1s_2^{-1}u_1+u_1s_2^{2}s_1^{-2}s_2s_1^{-1}s_2u_1+
u_1s_2s_1^{-2}s_2^{2}s_1^{-2}s_2^{2}u_1+\\&&+u_1s_2^{-1}s_1^2s_2^{-2}s_1^2s_2^{-2}u_1\\ \\
U'''&=&U''+u_1\grv^3+u_1\grv^{-3}\\ \\
U''''&=&U'''+u_1\grv^4+u_1\grv^{-4}\\ \\
U&=&U''''+u_1\grv^5+u_1\grv^{-5}.
\end{array}}$$

It is obvious that $U$ is a $u_1$-bi-module and that $U', U'', U'''$ and $U''''$ are $u_1-$ sub-bi-modules of $U$. Again, our goal is to 
prove that $H_5=U$ (theorem \ref{thh2}). As we explained in the proof of theorem \ref{th}, since $1\in U$ and $s_1U\subset U$ (by the definition of $U$), it is enough to prove that $s_2U\subset U$.
We notice that $$U=\sum_{k=1}^5u_1\grv^{\pm k}+\underbrace{u_1u_2u_1+u_1\text{``some elements of length 3''}u_1}_{\in U'}+\underbrace{u_1\text{``some elements of length 5''}u_1}_{\in U''}.$$ 
By the definition of $U'$ and $U''$ we have that $u_1\grv^{\pm1}\subset U'$ and $u_1\grv^{\pm2}\subset U''$. Therefore, in order to prove that $s_2U\subset U$ we only need to prove that $s_2u_1\grv^{\pm k}$ ($k=3,4,5$), $s_2U'$ and $s_2U''$ are subsets of $U$. 

The rest of this section is devoted to this proof (see proposition \ref{p2}, lemma \ref{oo}$(ii)$, proposition \ref{xx}$(i),(ii)$ and theorem \ref{thh2}). The reason we define also $U'''$ and $U''''$ is because, in order to prove that $s_2u_1\grv^k$ and $s_2u_1\grv^{-k}$ ($k=3,4,5$) are subsets of $U$, we want to ``replace'' inside the definition of $U$ the elements $\grv^k$ and $\grv^{-k}$ by some other elements modulo  $U'', U'''$ and $U''''$, respectively (see lemmas \ref{cc}, \ref{ll} and \ref{lol}).

Recalling that $\grF$ is the automorphism of $H_5$ as defined in section \ref{s}, we have the following lemma:
\begin{lem}The $u_1$-bi-modules $U', U'', U''', U''''$ and $U$ are stable under $\grF$.
	\label{r1}
\end{lem}
\begin{proof}We notice that $U',U'', U''', U''''$ and $U$ are of the form $$u_1s_2^{-2}s_1s_2^{-1}u_1+u_1s_2^{-1}s_1s_2^{-2}u_1+\sum u_1\grs u_1+\sum u_1\grs^{-1}u_1,$$ for some $\grs\in B_3$ satisfying $\grs^{-1}=\grF(\grs)$ and $\grs=\grF(\grs^{-1})$. Therefore, 
	we restrict ourselves to proving that  the elements $\grF(s_2^{-2}s_1s_2^{-1})=s_2^2s_1^{-1}s_2$ and $\grF(s_2^{-1}s_1s_2^{-2})=s_2s_1^{-1}s_2^2$ belong to $U'$. We expand $s_2^2$ as a linear combination of $s_2,1, s_2^{-1}, s_2^{-2}$ and $s_2^{-3}$ and by the definition of $U'$ and lemma \ref{lem1} we have to prove that the elements $s_2^{k}s_1^{-1}s_2$ and $s_2s_1^{-1}s_2^{k}$  are elements in $U'$, for $k=-3, -2$. Indeed, by using lemma \ref{lem1} we have: $s_2^{k}s_1^{-1}s_2=s_2^{k+1}(s_2^{-1}s_1^{-1}s_2)=(s_2^{k+1}s_1s_2^{-1})s_1^{-1}\in U'$ and  $s_2s_1^{-1}s_2^{k}=(s_2s_1^{-1}s_2^{-1})s_2^{k+1}=s_1^{-1}(s_2^{-1}s_1s_2^{k+1})\in U'$.
\end{proof}
From now on, we will use lemma \ref{lem1} without mentioning it.
\begin{prop} $u_2u_1u_2\subset U'.$
	\label{p1}
\end{prop}
\begin{proof}
	We have to prove that every element of the form $s_2^{\gra}s_1^{\grb}s_2^{\grg}$ belongs to $U'$, for $\gra,\grb,\grg\in\{-2,-1,0,1,2\}$. However, when $\gra\grb\grg=0$ the result is obvious. Therefore, we can assume that $\gra,\grb,\grg\in \{-2,-1,1,2\}.$ We continue the proof as in the proof of proposition \ref{prop1}, which is by distinguishing cases for $\gra$. However, by using lemma \ref{r1} we can assume that $\gra\in\{1,2\}$. We have:
	\begin{itemize}[leftmargin=*] 
		\item\underline{$\gra=1$}: 
		\begin{itemize}[leftmargin=*] 
			\item \underline{$\grg\in\{-1,1\}$}: The result follows from lemma \ref{lem1}, the braid relation and the definition of  $U'$.
			\item \underline{$\grg=-2$}:  $s_2s_1^{\grb}s_2^{-2}=(s_2s_1^{\grb}s_2^{-1})s_2^{-1}=s_1^{-1}(s_2^{\grb}s_1s_2^{-1})$. For $\grb\in\{1,-1,-2\}$ the result follows from lemma \ref{lem1} and the definition of $U'$. For $\grb=2$, we have $s_1^{-1}s_2^2s_1s_2^{-1}=s_1^{-1}s_2(s_2s_1s_2^{-1})=s_1^{-1}(s_2s_1^{-1}s_2)s_1\in U'$.
			\item \underline{$\grg=2$}:
			We need to prove that the element $s_2s_1^{\grb}s_2^{2}$ is inside $U'$. For $\grb\in\{-2,1\}$ the result is obvious by using the definition of $U'$ and the generalized braid relations. For $\grb=-1$ we have $s_2s_1^{-1}s_2^2=\grF(s_2^{-1}s_1s_2^{-2})\in\grF(U')\stackrel{\ref{r1}}{=}U'$. For $\grb=2$ we have $s_2s_1^2s_2^2= s_1^{-1}(s_1s_2s_1^2)s_2^2=s_1^{-1}s_2(s_2s_1s_2^3)=s_1^{-1}(s_2s_1^3s_2)s_1$. The result then follows from the case where $\grg=1$, if we expand $s_1^3$ as a linear combination of $s_1^2, s_1, 1, s_1^{-1}, s_1^{-2}$.
		\end{itemize}
		\item \underline {$\gra=2$}:
		\begin{itemize}[leftmargin=*] 
			\item \underline{$\grg=-1$}:  $s_2^2s_1^{\grb}s_2^{-1}=s_2(s_2s_1^{\grb}s_2^{-1})=(s_2s_1^{-1}s_2^{\grb})s_1\in U'$ (case where $\gra=1$).
			\item \underline{$\grg=2$}: We only have to prove the cases where $\grb\in\{-1,1\}$, since the cases where $\grb\in \nolinebreak \{2,-2\}$ follow from the definition of $U'$. We have $s_2^2s_1s_2^2=(s_2^2s_1s_2)s_2=s_1\grv\in U'$. Moreover,
			$s_2^2s_1^{-1}s_2^2=s_1^{-1}(s_1s_2^2s_1^{-1})s_2^2=s_1^{-1}\grF(s_2s_1^{-2}s_2^{-3})$. The result follows from the case where $\gra=1$ and lemma
			\ref{r1}, if we expand $s_2^{-3}$ as  a linear combination of $s_2^{-2}, s_2^{-1}, 1, s_2, s_2^{2}$.
			\item \underline{$\grg=1$}: We  have to check the cases where $\grb\in\{-2,-1,2\}$, since the case where $\grb=1$ is a direct result from the generalized braid relations. However,   $s_2^2s_1^{-1}s_2=\grF(s_2^{-2}s_1s_2^{-1})\in\grF(U')\stackrel{\ref{r1}}{=}U'$. Hence, it remains to prove the cases where $\grb\in \{-2,2\}$. 
			We have $s_2^2s_1^{-2}s_2=s_2^3(s_2^{-1}s_1^{-2}s_2)=s_1(s_1^{-1}s_2^3s_1)s_2^{-2}s_1^{-1}=
			s_1(s_2s_1^3s_2^{-3})s_1^{-1}$. The latter is an element in $U'$, if we expand $s_1^3$ and $s_2^{-3}$ as  linear combinations of $s_1^2, s_1, 1, s_1^{-1}, s_1^{-2}$ and $s_2^{-2}, s_2^{-1}, 1, s_2, s_2^{2}$, respectively and use the case where $\gra=1$. Moreover, 
			$s_2^2s_1^2s_2=
			s_2^2s_1(s_1s_2s_1)s_1^{-1}=(s_2^2s_1s_2)s_1s_2s_1^{-1}=s_1(s_2s_1^3s_2)s_1.$ The result follows again from the case where $\gra=1$, if we expand $s_1^3$ as a linear combination of $s_1^2, s_1, 1, s_1^{-1}, s_1^{-2}$. 
			
			\item\underline{$\grg=-2$}: We need to prove that $s_2^2s_1^{\grb}s_2^{-2}\in U'$. For $\grb=2$ the result follows from the definition of $U'$. For $\grb\in\{1,-1\}$ we have: $s_2^2s_1s_2^{-2}=s_2^2(s_1s_2^{-2}s_1^{-1})s_1=(s_2s_1^{-2}s_2)s_1\in U'$.
			$s_2^2s_1^{-1}s_2^{-2}=s_2(s_2s_1^{-1}s_2^{-1})s_2^{-1}=(s_2s_1^{-1}s_2^{-1})s_1s_2^{-1}=
			s_1^{-1}(s_2^{-1}s_1^2s_2^{-1})\in U'$. 
			It remains to prove the case where $\grb=-2$. We recall that $\grv =s_2s_1^2s_2$ and we have:
			$ s_2^2s_1^{-2}s_2^{-2}=s_1^{-1}(s_1s_2^2s_1^{-1})s_1^{-1}s_2^{-2}=s_1^{-1}s_2^{-2}\grv s_1^{-1}s_2^{-2}=s_1^{-1}s_2^{-2} s_1^{-1}\grv s_2^{-2}=
			s_1^{-1}s_2^{-2}s_1^{-1}(s_2s_1^2s_2^{-1})= s_1^{-1}(s_2^{-2}s_1^{-2}s_2^2)s_1. $ The result follows from the definition of $U'$.
			\qedhere
		\end{itemize}
	\end{itemize}
\end{proof}
From now on, in order to make it easier for the reader to follow the calculations, we will underline the elements belonging to $u_1u_2u_1u_2u_1$ and we will use immediately the fact that these elements belong to $U'$ (see proposition \ref{p1}).
\begin{lem}
	\mbox{}  
	\vspace*{-\parsep}
	\vspace*{-\baselineskip}\\ 
	\begin{itemize}[leftmargin=0.6cm]
		\item [(i)]$s _2u_1s_2u_1s_2u_1\subset\grv^2u_1+u_1u_2u_1u_2u_1\subset U''$.  
		\item [(ii)]$s_2\grv^2u_1=s_1s_2s_1^4s_2s_1^3s_2u_1\subset U''.$
	\end{itemize}  
	\label{oo}
\end{lem}
\begin {proof}
We recall that $\grv=s_2s_1^2s_2$.
\begin{itemize}[leftmargin=0.6cm]
	\item[(i)] The fact that $\grv^2u_1+u_1u_2u_1u_2u_1\subset U''$ follows directly from the definition of $U''$ and  proposition \ref{p1}. For the rest of the proof, we use the definition of $u_1$ and we have that  $s_2u_1s_2u_1s_2u_1=s_2u_1s_2(R_5+R_5s_1^{-1}+R_5s_1+R_5
	s_1^{2}+R_5s_1^3)s_2u_1\subset \underline{s_2u_1s_2^2u_1}+s_2u_1s_2s_1^{-1}s_2u_1
	+\underline{s_2u_1(s_2s_1s_2)u_1}+
	s_2u_1\grv+s_2u_1s_2s_1^{3}s_2u_1$.
However, $s_2u_1\grv=\underline{s_2\grv u_1}$ and $s_2u_1s_2s_1^{-1}s_2u_1=s_2u_1(s_1s_2s_1^{-1})s_2u_1= (s_2u_1s_2^{-1})s_1s_2^2u_1=\underline{s_1^{-1}u_2s_1^2s_2^2u_1}$. Therefore, it is enough to prove that $s_2u_1s_2s_1^3s_2u_1\subset\grv^2u_1+u_1u_2u_1u_2u_1$. For this purpose, we use again the definition of $u_1$ and we have:
	
	$\small{\begin{array}[t]{lcl}
	s_2u_1s_2s_1^3s_2u_1
	&\subset&
	s_2(R_5+R_5s_1+R_5s_1^{-1}+R_5s_1^2+R_5s_1^3)s_2s_1^{3}s_2u_1\\
	&\subset& \underline{s_2^2s_1^3s_2u_1}+
	\underline{s_2(s_1s_2s_1^{3})s_2u_1}+
	s_2(s_1^{-1}s_2s_1)s_1^{2}s_2u_1+\grv s_1^{3}s_2u_1+s_2s_1^{2}(s_1s_2s_1^{3})s_2u_1\\
	&\subset& \underline{s_2^2s_1(s_2^{-1}s_1^2s_2)u_1}+\underline{s_1^{3}\grv s_2u_1}+s_2s_1^2s_2^2(s_2s_1s_2^2)u_1+u_1u_2u_1u_2u_1\\
	&\subset&\grv^2u_1+u_1u_2u_1u_2u_1.
	\end{array}}$
	\item[(ii)] We have that $s_2\grv^2=s_1(s_1^{-1}s_2^2s_1)(s_1s_2s_1)s_1^{-1}\grv=s_1s_2s_1^4(s_1^{-1}s_2s_1)s_1^{-2}\grv=s_1s_2s_1^4s_2s_1s_2^{-1}s_1^{-2}\grv=s_1s_2s_1^4s_2s_1s_2^{-1}\grv s_1^{-2}=s_1s_2s_1^4s_2s_1^3s_2s_1^{-2}$. Therefore, $s_2\grv^2u_1
	\subset
	u_1s_2u_1s_2u_1s_2u_1.$	The fact that $u_1s_2u_1s_2u_1s_2u_1\subset U''$ follows immediately from (i).
	\qedhere
\end{itemize}
\end{proof}
\begin{prop}\mbox{}  
	\vspace*{-\parsep}
	\vspace*{-\baselineskip}\\ 
	\begin{itemize}[leftmargin=0.6cm]
		\item [(i)]$u_2u_1s_2^{-1}s_1s_2^{-1}\subset u_1\grv^{-2}+R_5s_2^{-2}s_1^2s_2^{-1}s_1s_2^{-1}+u_1u_2u_1u_2u_1\subset U''.$
		\item[(ii)] $u_2u_1s_2s_1^{-1}s_2\subset u_1 \grv^{2}+R_5s_2^{2}s_1^{-2}s_2s_1^{-1}s_2+u_1u_2u_1u_2u_1\subset U''.$
	\end{itemize}
	\label{l2}
\end{prop}
\begin{proof}
	We restrict ourselves to proving $(i)$, since $(ii)$ follows from $(i)$ by applying $\grF$ (see lemma \ref{r1}). By the definition of $U''$ and by proposition \ref{p1} we have that  $u_1\grv^{-2}+R_5s_2^{-2}s_1^2s_2^{-1}s_1s_2^{-1}+u_1u_2u_1u_2u_1\subset U''$. Therefore, it remains to prove that  $u_2u_1s_2^{-1}s_1s_2^{-1}\subset u_1\grv^{-2}+R_5s_2^{-2}s_1^2s_2^{-1}s_1s_2^{-1}+u_1u_2u_1u_2u_1.$
	
	By he definition of $u_1$ we have that 
		$u_2u_1s_2^{-1}s_1s_2^{-1}=u_2(R_5+R_5s_1+R_5s_1^{-1}+R_5s_1^{-2}
		+R_5s_1^{2})s_2^{-1}s_1s_2^{-1}
		\subset \underline{u_2s_1s_2^{-1}}+
		u_2s_1s_2^{-1}s_1s_2^{-1}+\underline{u_2(s_1^{-1}s_2^{-1}s_1)s_2^{-1}}+u_2s_1^{-2}s_2^{-1}s_1s_2^{-1}
		+u_2s_1^{2}s_2^{-1}s_1s_2^{-1}$.
	We notice that $u_2s_1s_2^{-1}s_1s_2^{-1}=u_2(s_2s_1s_2^{-1})s_1s_2^{-1}=\underline{u_2s_1^{-1}(s_2s_1^2s_2^{-1})}$.
	Therefore, we only have to prove that $u_2s_1^{-2}s_2^{-1}s_1s_2^{-1}$ and 
	$u_2s_1^{2}s_2^{-1}s_1s_2^{-1}$ are subsets of $ u_1\grv^{-2}+R_5s_2^{-2}s_1^2s_2^{-1}s_1s_2^{-1}+u_1u_2u_1u_2u_1$. We have: \\ 
$\small{\begin{array}[t]{lcl}
		u_2s_1^{-2}s_2^{-1}s_1s_2^{-1}
	&\subset&(R_5+R_5s_2+R_5s_2^{-1}+
		R_5s_2^2+R_5s_2^3)s_1^{-2}s_2^{-1}s_1s_2^{-1}\\
		&\subset&\underline{R_5s_1^{-2}s_2^{-1}s_1s_2^{-1}}+\underline{
			R_5(s_2s_1^{-2}s_2^{-1})s_1s_2^{-1}}+R_5\grv^{-1}s_1s_2^{-1}+
		R_5s_2(s_2s_1^{-2}s_2^{-1})s_1s_2^{-1}+\\&&+
		R_5s_2^2(s_2s_1^{-2}s_2^{-1})s_1s_2^{-1}\\
		
		&\subset&\underline{R_5s_1\grv^{-1}s_2^{-1}}+R_5(s_2s_1^{-1}s_2^{-1})s_2^{-1}s_1^2s_2^{-1}+
		R_5s_2(s_2s_1^{-1}s_2^{-1})s_2^{-1}s_1^2s_2^{-1}+u_1u_2u_1u_2u_1\\
		&\subset&R_5s_1^{-1}s_2^{-1}s_1s_2^{-1}s_1^2s_2^{-1}+R_5(s_2s_1^{-1}s_2^{-1})s_1s_2^{-1}s_1^2s_2^{-1}+u_1u_2u_1u_2u_1\\
		&\subset&\grF(u_1s_2u_1s_2u_1s_2)+u_1u_2u_1u_2u_1.
		\end{array}}$\\ \\
However, by lemma \ref{oo}(i) we have that $\grF(u_1s_2u_1s_2u_1s_2)\subset \grF(\grv^2u_1+u_1u_2u_1u_2u_1)=\grv^{-2}u_1+u_1u_2u_1u_2u_1$. Therefore, 
$ u_2s_1^{-2}s_2^{-1}s_1s_2^{-1}\subset 	\grv^{-2}u_1+u_1u_2u_1u_2u_1$.	
By using analogous calculations, we have: \\
$\small{\begin{array}[t]{lcl}
		u_2s_1^{2}s_2^{-1}s_1s_2^{-1}	&\subset&(R_5+R_5s_2
	+R_5s_2^{-1}+R_5s_2^2+R_5s_2^{-2})s_1^{2}s_2^{-1}s_1s_2^{-1}\\
	&\subset& \underline{R_5s_1^2s_
		2^{-1}s_1s_2^{-1}}+\underline{R_5(s_2s_1^2s_2^{-1})s_1s_2^{-1}}+
	R_5s_2^{-1}s_1^3(s_1^{-1}s_2^{-1}s_1)s_2^{-1}+
	R_5s_2(s_2s_1^2s_2^{-1})s_1s_2^{-1}+\\&&+R_5s_2^{-2}s_1^2s_2^{-1}s_1s_2^{-1}\\
	&\subset&
	\underline{R_5(s_2^{-1}s_1^3s_2)s_1^{-1}s_2^{-2}}+
	R_5s_2s_1^{-1}s_2^{2}s_1^2s_2^{-1}+R_5s_2^{-2}s_1^2s_2^{-1}s_1s_2^{-1}+u_1u_2u_1u_2u_1.
	\end{array}}$\\ \\
It is enough  to prove that $s_2s_1^{-1}s_2^{2}s_1^2s_2^{-1}\subset u_1u_2u_1u_2u_1$. Indeed, we have: $s_2s_1^{-1}s_2^2s_1^2s_2^{-1}=s_1^{-1}(s_1s_2s_1^{-1})s_2(s_2s_1^2s_2^{-1})=\underline{s_1^{-1}s_2^{-1}(s_1s_2^2s_1^{-1})s_2^2s_1}$.

	\qedhere
\end{proof}
We can now prove a lemma that helps us to ``replace''  inside the definition of $U'''$ the element $\grv^3$ with the element $s_2s_1^3s_2^2s_1^2s_2^2$ modulo $U''$.
\begin{lem}$s_2s_1^3s_2^2s_1^2s_2^2\in u_1s_2u_1s_2s_1^3s_2u_1+u_1s_2^2s_1^3s_2s_1^{-1}s_2u_1+
	u_1u_2u_1u_2u_1+ u_1^{\times}\grv^3\subset u_1^{\times}\grv^3+\nolinebreak U''.$
	\label{cc}
\end{lem}
\begin{proof}
The fact that 	$u_1s_2u_1s_2s_1^3s_2u_1+u_1s_2^2s_1^3s_2s_1^{-1}s_2u_1+
u_1u_2u_1u_2u_1+ u_1^{\times}\grv^3$ is a subset of $u_1^{\times}\grv^3+U''$ follows  from lemma \ref{oo}$(i)$ and propositions \ref{l2}(ii) and \ref{p1}. 
For the rest of the proof, 	we have: 	
		$s_2s_1^3s_2^2s_1^2s_2^2=s_2s_1^2(s_1s_2^2s_1^{-1})s_1^2(s_1s_2^2s_1^{-1})s_1=s_2s_1^2s_2^{-2}\grv s_1^2s_2^{-1}s_1(s_1s_2s_1^{-1})s_1^2=
		s_2s_1^2s_2^{-2}s_1^2\grv s_2^{-1}s_1s_2^{-1}(s_1s_2s_1^{-1})s_1^3=s_2s_1^2s_2^{-2}s_1^2s_2s_1^3s_2^{-2}s_1s_2s_1^3=s_2s_1^2s_2^{-3}\bold{\boldsymbol{\grv} s_1^3s_2^{-2}}s_1s_2s_1^3$.
		However, $\bold{\boldsymbol{\grv} s_1^3s_2^{-2}}=s_1^3\grv s_2^{-2}=s_1^3(s_2s_1^2s_2^{-1})=s_1^2s_2^2s_1$ and, hence, $s_2s_1^3s_2^2s_1^2s_2^2=s_2s_1^2s_2^{-3}s_1^2s_2^2s_1^2s_2s_1^3$.
		
		 Our goal now is to prove that the element $s_2s_1^2s_2^{-3}s_1^2s_2^2s_1^2s_2s_1^3$ is inside $ u_1s_2u_1s_2s_1^3s_2u_1+u_1s_2^2s_1^3s_2s_1^{-1}s_2u_1+
		u_1u_2u_1u_2u_1+ u_1^{\times}\grv^3$. For this purpose we expand $s_2^{-3}$ as a linear combination of $s_2^{-2}$, $s_2^{-1}$, 1, $s_2$ and $s_2^2$, where the coefficient of $s_2^2$ is invertible, and we have that  
	$s_2s_1^2s_2^{-3}s_1^2s_2^2s_1^2s_2s_1^3\in s_2s_1^2s_2^{-2}s_1^2s_2^2s_1^2s_2u_1+s_2s_1^2s_2^{-1}s_1^2s_2^2s_1^2s_2u_1+
	s_2s_1^4s_2^2s_1^2s_2u_1+s_2\grv^2u_1+u_1^{\times}\grv^3$.
	However, by lemma \ref{oo}(ii) we have that $s_2\grv^2u_1\subset u_1s_2u_1s_2s_1^3s_2u_1$. Moreover, $s_2s_1^4s_2^2s_1^2s_2u_1=s_2s_1^5(s_1^{-1}s_2^2s_1)(s_1s_2s_1)u_1\subset u_1s_2u_1s_2s_1^3s_2u_1$. It remains to prove that the elements
	$s_2s_1^2s_2^{-2}s_1^2s_2^2s_1^2s_2$ and $s_2s_1^2s_2^{-1}s_1^2s_2^2s_1^2s_2$
	are inside $u_1s_2u_1s_2s_1^3s_2u_1+u_1s_2^2s_1^3s_2s_1^{-1}s_2u_1+
	u_1u_2u_1u_2u_1$.
	
	 On one hand,  we have
	$s_2s_1^2s_2^{-2}s_1^2s_2^2s_1^2s_2=s_2s_1^3(s_1^{-1}s_2^{-2}s_1)s_1s_2\grv=s_2s_1^3s_2s_1^{-1}(s_1^{-1}s_2^{-1}s_1)s_2\grv=s_2s_1^3s_2^2(s_2^{-1}s_1^{-1}s_2)s_1^{-1}\grv=s_2s_1^3s_2^2s_1s_2^{-1}s_1^{-2}\grv=s_2s_1^3s_2^2s_1s_2^{-1}\grv s_1^{-2}=s_2s_1^3s_2^2s_1^3s_2s_1^{-2}$, meaning that the element $s_2s_1^2s_2^{-2}s_1^2s_2^2s_1^2s_2$ is inside $s_2s_1^3s_2^2u_1s_2u_1$. On the other hand, 
	$s_2s_1^2s_2^{-1}s_1^2s_2^2s_1^2s_2
=s_2s_1^2(s_2^{-1}s_1^2s_2)\grv=s_2s_1^3s_2^2s_1^{-1}\grv=s_2s_1^3s_2^2\grv s_1^{-1}=s_2s_1^3s_2^3s_1^2s_2s_1^{-1}$ and, if we expand $s_2^3$ as a linear combination of $s_2^2$, $s_2$, 1, $s_2^{-1}$ and $s_2^{-2}$, we have that $s_2s_1^3s_2^3s_1^2s_2s_1^{-1}\in s_2s_1^3s_2^{2}s_1^2s_2u_1+s_2s_1^3\grv u_1+
\underline{s_2s_1^5s_2u_1}+\underline{
	(s_2s_1^3s_2^{-1})s_1^2s_2u_1}+\underline{(s_2s_1^3s_2^{-1})(s_2^{-1}s_1^2s_2)u_1}\subset 
	s_2s_1^3s_2^{2}u_1s_2u_1+\underline{s_2\grv u_1}+u_1u_2u_1u_2u_1$, meaning that the element $s_2s_1^2s_2^{-1}s_1^2s_2^2s_1^2s_2$ is inside 	$s_2s_1^3s_2^{2}u_1s_2u_1+u_1u_2u_1u_2u_1$. As a result, in order to finish the proof, it will be sufficient to show that $s_2s_1^3s_2^{2}u_1s_2u_1$ is a subset of
	$u_1s_2u_1s_2s_1^3s_2u_1+u_1s_2^2s_1^3s_2s_1^{-1}s_2u_1+
	u_1u_2u_1u_2u_1$. Indeed, we have: 
$$\small{\begin{array}[t]{lcl}
			s_2s_1^3s_2^{2}u_1s_2u_1		&\subset&s_2s_1^3s_2^2(R_5s_1^2+R_5s_1+R_5+R_5s_1^{-1}+R_5s_1^{-2})s_2u_1\\
			&\subset&s_2s_1^3s_2^2s_1^2s_2u_1+\underline{s_2s_1^3(s_2^2s_1s_2)u_1}+\underline{s_2s_1^3s_2^3u_1}
			+s_2s_1^2(s_1s_2^2s_1^{-1})s_2u_1+
			s_2s_1^2(s_1s_2^2s_1^{-1})s_1^{-1}s_2u_1\\
			&\subset&s_2s_1^4(s_1^{-1}s_2^2s_1)(s_1s_2s_1)u_1+\underline{
				(s_2s_1^2s_2^{-1})s_1^2s_2^2u_1}+
			(s_2s_1^2s_2^{-1})s_1^2s_2s_1^{-1}s_2u_1+u_1u_2u_1u_2u_1\\
			&\subset& u_1s_2u_1s_2s_1^3s_2u_1+u_1s_2^2s_1^3s_2s_1^{-1}s_2u_1+u_1u_2u_1u_2u_1.
			\end{array}}$$
		
\end{proof} 
\begin{prop}
	\mbox{} 
	\vspace*{-\parsep}
	\vspace*{-\baselineskip}\\
	\begin{itemize}[leftmargin=0.6cm]
		\item [(i)]$s_2u_1u_2u_1u_2 \subset U'''$.
		\item [(ii)]$s_2^{-1}u_1u_2u_1u_2 \subset U'''$.
	\end{itemize}
	\label{p2}
\end{prop}
\begin{proof}By lemma \ref{r1}, we only have to prove $(i)$, since $(ii)$ is a consequence of $(i)$ up to applying $\grF$. We know that $u_2u_1u_2 \subset U'$ (proposition \ref{p1}) hence it is enough to prove that $s_2U'\subset U'''$. Set 
	$$\small{\begin{array}{lcl}V&=&u_1u_2u_1+\grv u_1+\grv^{-1}u_1+u_1s_2^{-1}s_1^2s_2^{-1}u_1+
	u_1s_2^{-1}s_1s_2^{-1}u_1+u_1s_2s_1^{-1}s_2u_1+u_1s_2^{-2}s_1s_2^
	{-1}u_1+\\&&+u_1s_2^{-1}s_1^2s_2^{-2}u_1+u_1s_2^{-1}s_1s_2^{-2}u_1+u_1s_2s_1^{-2}s_2u_1+
	u_1s_2^{-2}s_1^{-2}s_2^{-2}u_1+
	u_1s_2^{-2}s_1^{-2}s_2^{2}u_1.
	\end{array}}$$
	We notice that 
	$$\begin{array}{lcl}
	U'&=&V+u_1s_2s_1^{-2}s_2^{2}u_1+
	u_1s_2^{2}s_1^{-2}s_2^{2}u_1+u_1s_2^{2}s_1^{2}s_2^{2}u_1+
	u_1s_2^{-2}s_1^{2}s_2^{-2}u_1+u_1s_2^{2}s_1^{2}s_2^{-2}u_1.
	\end{array}$$
	Therefore, in order to prove that $s_2U'\subset U'''$, we will prove first that $s_2V\subset U'''$ and then we will check the other five cases separately. 
	We have: 
	$$\small{\begin{array}{lcl}s_2V&\subset&
		\underline{s_2u_1u_2u_1}+\underline{s_2\grv u_1}+\underline{s_2\grv^{-1} u_1}+\underline{(s_2u_1s_2^{-1})u_1u_2u_1}+s_2u_1s_2u_1s_2+s_2u_1s_2^{-2}s_1s_2^{-1}+\\&&+
		s_2u_1s_2^{-2}s_1^{-2}s_2^{-2}u_1+
		s_2u_1s_2^{-2}s_1^{-2}s_2^{2}u_1+U'''
		\end{array}}$$
	However, by proposition \ref{oo}(i) we have that $s_2u_1s_2u_1s_2\subset U''\subset U'''$. It remains to prove that $A:=s_2u_1s_2^{-2}s_1s_2^{-1}+
	s_2u_1s_2^{-2}s_1^{-2}s_2^{-2}u_1+
	s_2u_1s_2^{-2}s_1^{-2}s_2^{2}u_1$ is a subset of $U'''$. We have: 
	$$\small{\begin{array}{lcl}
		A&=&s_2u_1s_2^{-2}s_1s_2^{-1}+s_2u_1s_2^{-2}s_1^{-2}s_2^{-2}u_1+
		s_2u_1s_2^{-2}s_1^{-2}s_2^{2}u_1\\
		&=&(s_2u_1s_2^{-1})s_2^{-1}s_1s_2^{-1}u_1+
		(s_2u_1s_2^{-1})s_2^{-1}s_1^{-2}s_2^{-2}u_1+
		(s_2u_1s_2^{-1})s_2^{-1}s_1^{-2}s_2^2u_1\\
		&=&s_1^{-1}u_2(s_1s_2^{-1}s_1^{-1})s_1^2s_2^{-1}u_1+s_1^{-1}u_2(s_1s_2^{-1}s_1^{-1})s_1^{-1}s_2^{-2}u_1+
		s_1^{-1}u_2s_1(s_2^{-1}s_1^{-2}s_2)s_2u_1\\
		&=&s_1^{-1}u_2s_1^{-1}(s_2s_1^2s_2^{-1})u_1+s_1^{-1}u_2s_1^{-1}(s_2s_1^{-1}s_2^{-1})s_2^{-1}u_1+s_1^{-1}u_2s_1^2s_2^{-1}(s_2^{-1}s_1^{-1}s_2)u_1\\
		&\subset&u_1(u_2u_1s_2^{-1}s_1s_2^{-1})u_1.
		\end{array}}$$
	By proposition \ref{l2} we have then $A\subset U'''$ and, hence, we proved that
	 \begin{equation}s_2V\subset U'''\label{7}\end{equation}
	In order to finish the proof that $s_2U'\subset U''$,  it will be sufficient to prove that $u_1s_2s_1^{-2}s_2^{2}u_1$, 
	$u_1s_2^{2}s_1^{-2}s_2^{2}u_1$, $u_1s_2^{2}s_1^{2}s_2^{2}u_1$,
	$u_1s_2^{-2}s_1^{2}s_2^{-2}u_1$ and $u_1s_2^{2}s_1^{2}s_2^{-2}u_1$ are subsets of $U'''$. 
	\begin{itemize}[leftmargin=0.6cm]
		\item[C1.] We will prove that $s_2u_1s_2s_1^{-2}s_2^2u_1\subset U'''$. We expand $s_2^2$ as a linear combination of $s_2$, 1 $s_2^{-1}$, $s_2^{-2}
$ and $s_2^{-3}$ and we have that 		
		$s_2u_1s_2s_1^{-2}s_2^2u_1\subset s_2u_1s_2s_1^{-2}s_2u_1+\underline{s_2u_1s_2u_1}+\underline{s_2u_1(s_2s_1^{-2}s_2^{-1})u_1}+
		s_2u_1(s_2s_1^{-2}s_2^{-1})s_2^{-1}u_1+
		s_2u_1s_2s_1^{-2}s_2^{-3}u_1
	\subset s_2u_1s_2s_1^{-2}s_2^{-3}u_1+s_2V+U'''$ and, hence, 
		by relation (\ref{7}) we have that $s_2u_1s_2s_1^{-2}s_2^2u_1\subset s_2u_1s_2s_1^{-2}s_2^{-3}u_1+U'''$. Therefore, it will be sufficient to prove that $s_2u_1s_2s_1^{-2}s_2^{-3}u_1\subset U'''$. We use the definition of $u_1$ and we have:
		
		$\small{\begin{array}[t]{lcl}
			s_2u_1s_2s_1^{-2}s_2^{-3}u_1
			&\subset&s_2(R_5+R_5s_1+R_5s_1^{-1}+R_5s_1^2+R_5s_1^{3})s_2s_1^{-2}s_2^{-3}u_1\\
			&\subset&\underline{s_2^2s_1^{-2}s_2^{-3}u_1}+\underline{(s_2s_1s_2)s_1^{-2}s_2^{-3}u_1}+
			s_2s_1^{-1}s_2s_1^{-2}s_2^{-3}u_1+\grv s_1^{-2}s_2^{-3}u_1+\\&&+s_2s_1^{3}s_2s_1^{-2}s_2^{-3}u_1\\
			&\subset&s_1^{-1}(s_1s_2s_1^{-1})s_2s_1^{-2}s_2^{-3}u_1+\underline{s_1^{-2}\grv s_2^{-3}u_1}+
			s_2s_1^{2}(s_1s_2s_1^{-1})s_1^{-1}s_2^{-3}u_1+U'''\\
			&\subset&s_1^{-1}s_2^{-1}(s_1s_2^2s_1^{-1})s_1^{-1}s_2^{-3}u_1+
			(s_2s_1^{2}s_2^{-1})s_1s_2s_1^{-1}s_2^{-3}u_1+U'''\\
			&\subset&s_1^{-1}s_2^{-2}s_1^2(s_2s_1^{-1}s_2^{-1})s_2^{-2}u_1+
			s_1^{-1}s_2^2s_1(s_1s_2s_1^{-1})s_2^{-3}u_1+
			U'''\\
			&\subset&s_1^{-1}s_2^{-3}(s_2s_1s_2^{-1})s_1s_2^{-2}u_1+s_1^{-1}s_2(s_2s_1s_2^{-1})s_1s_2^{-2}u_1+U'''\\
			&\subset&s_1^{-1}s_2^{-3}s_1^{-1}(s_2s_1^2s_2^{-1})s_2^{-1}u_1+s_1^{-1}s_2s_1^{-1}(s_2s_1^2s_2^{-1})s_2^{-1}u_1+U'''\\
			&\subset&
			s_1^{-1}s_2^{-3}s_1^{-2}s_2(s_2s_1s_2^{-1})u_1+s_1^{-1}s_2s_1^{-2}s_2(s_2s_1s_2^{-1})u_1+U'''\\
			&\subset&u_1(u_2u_1s_2s_1^{-1}s_2)u_1+U'''.
			\end{array}}$
		
		The result follows from proposition \ref{l2}(ii).
		\item[C2.] We will prove that $s_2u_1s_2^2s_1^{-2}s_2^2u_1\subset U'''$. For this purpose, we expand $u_1$ as $R_5+R_5s_1+R_5s_1^4+R_5s_1^2+R_5s_1^{-2}$ and we have that $s_2u_1s_2^2s_1^{-2}s_2^2u_1\subset 
	\underline{s_2^3s_1^{-2}s_2^2u_1}+\underline{
		(	s_2s_1s_2^2)s_1^{-2}s_2^2u_1}+s_2s_1^4s_2^2s_1^{-2}s_2^2u_1
		+s_2s_1^2s_2^2s_1^{-2}s_2^2u_1+s_2s_1^{-2}s_2^2s_1^{-2}s_2^2u_1$.
		By the definition of $U'''$ we have that $s_2s_1^{-2}s_2^2s_1^{-2}s_2^2u_1\subset U'''$. Therefore, it remains to  prove that $s_2s_1^4s_2^2s_1^{-2}s_2^2u_1
		+s_2s_1^2s_2^2s_1^{-2}s_2^2u_1\subset U'''$. We notice that

		$\small{\begin{array}{lcl}s_2s_1^4s_2^2s_1^{-2}s_2^2u_1
		+s_2s_1^2s_2^2s_1^{-2}s_2^2u_1&\subset&
		s_2s_1^3(s_1s_2^2s_1^{-1})s_1^{-1}s_2^2u_1
		+\grv (s_2s_1^{-2}s_2^{-1})s_2^3u_1\\&\subset&(s_2s_1^3s_2^{-1})s_1(s_1s_2s_1^{-1})s_2^2u_1+
		\grv s_1^{-2}(s_1s_2^{-2}s_1^{-1})s_1^2s_2^3u_1\\
		&\subset& s_1^{-1}s_2^3s_1^2s_2^{-1}s_1s_2^3u_1+\underline{s_1^{-2}\grv s_2^{-1}s_1^{-2}s_2s_1^2s_2^3u_1}
		\end{array}}$
		
	Therefore, we have to prove that the element $s_2^3s_1^2s_2^{-1}s_1s_2^3$ is inside $U'''$. For this purpose, we expand $s_2^3$ as a linear combination  of 
		$s_2^2$, $s_2$, 1 $s_2^{-1}$ and $s_2^{-2}$ and we have:

		$\small{\begin{array}[t]{lcl}
			s_2^3s_1^2s_2^{-1}s_1s_2^3
			&\in&
			R_5s_2^3s_1^{3}(s_1^{-1}s_2^{-1}s_1)s_2^2+\underline{R_5s_2^3s_1^2(s_2^{-1}s_1s_2)}+\underline
			{R_5s_2^3s_1^2s_2^{-1}}+R_5s_2^3s_1^2s_2^{-1}s_1s_2^{-1}+\\&&+
			R_5s_1^{-1}(s_1s_2^2s_1^{-1})s_1(s_2s_1^2s_2^{-1})s_1s_2^{-2}\\
			&\in&u_2u_1s_2s_1^{-1}s_2+u_2u_1s_2^{-1}s_1s_2^{-1}u_1+u_1s_2^{-1}s_1^2s_2^3s_1^2s_2^{-2}+U'''.
			
			\end{array}}$
			
	However, by proposition 	\ref{l2} we have that 	$u_2u_1s_2s_1^{-1}s_2$ and $u_2u_1s_2^{-1}s_1s_2^{-1}$ are subsets of $U'''$. Therefore, we only need to prove that the element $s_2^{-1}s_1^2s_2^3s_1^2s_2^{-2}$ is inside $U'''$. We expand $s_2^3$ as a linear combination  of 
	$s_2^2$, $s_2$, 1 $s_2^{-1}$ and $s_2^{-2}$ and we have that $s_2^{-1}s_1^2s_2^3s_1^2s_2^{-2}\in \grF(s_2V)+\underline{s_2^{-1}(s_1^2s_2s_1)s_1s_2^{-2}}+      \grF(s_2u_1s_2s_1^{-2}s_2^2)+R_5s_2^{-1}s_1^{2}s_2^{-2}s_1^{2}s_2^{-2}$. However, by the definition of $U'''$ we have that $s_2^{-1}s_1^{2}s_2^{-2}s_1^{2}s_2^{-2}\in U'''$.
			Moreover, by relation (\ref{7}) and by the previous case (case C1) we have that $\grF(s_2V)+\grF(s_2u_1s_2s_1^{-2}s_2^2)\subset \grF(U''')\stackrel{\ref{r1}}\subset U'''.$ 
			
		\item[C3.] We will prove that $s_2u_1s_2^2s_1^{2}s_2^2\subset U'''$. For this purpose, we expand $u_1$ as $R_5+R_5s_1+R_5s_1^{-1}+
		R_5s_1^2+R_5s_1^3$ and we have $s_2u_1s_2^2s_1^{2}s_2^2\subset \underline{s_2^3s_1^2s_2^2u_1}+\underline{(s_2s_1s_2^2)s_1^2s_2^2u_1}+ s_2s_1^{-1}s_2^2s_1^{2}s_2^2u_1+s_2s_1^{2}s_2^2s_1^2s_2^2u_1+s_2s_1^3s_2^2s_1^{2}s_2^2$. However, be lemma \ref{cc} we have that $s_2s_1^3s_2^2s_1^{2}s_2^2\subset u_1\grv^3+U''\subset U'''$. Therefore, it remains to  prove that $s_2s_1^{-1}s_2^2s_1^{2}s_2^2u_1+s_2s_1^{2}s_2^2s_1^2s_2^2u_1\subset U'''$. We have: 
		
		$\small{\begin{array}[t]{lcl}s_2s_1^{-1}s_2^2s_1^{2}s_2^2u_1+s_2s_1^{2}s_2^2s_1^2s_2^2u_1&=&
		s_2^2(s_2^{-1}s_1^{-1}s_2)s_2s_1^{2}s_2^2u_1+s_1^{-1}s_1\grv^2s_2\\
		&=&s_2^2s_1(s_2^{-1}s_1^{-1}s_2)s_1^{2}s_2^2u_1+
		s_1^{-1}\grv^2s_1s_2\\
		&=&s_2^2s_1^2(s_2^{-1}s_1s_2)s_2u_1+s_1^{-1}s_2s_1^2s_2^2s_1^2(s_2s_1s_2)\\
		&\subset& u_2u_1s_2s_1^{-1}s_2u_1+u_1s_2s_1^2s_2^2s_1^3s_2u_1.
		
		\end{array}}$.
	
	By lemma \ref{l2}(ii) it will be sufficient to prove that $s_2s_1^2s_2^2s_1^3s_2\in U'''$. We expand $s_2^3$ as a linear combination  of 
	$s_2^2$, $s_2$, 1 $s_2^{-1}$ and $s_2^{-2}$ and we have: 
		
		$\small{\begin{array}[t]{lcl}
		s_2s_1^2s_2^2s_1^3s_2&\in& R_5\grv^2+\underline{R_5s_2s_1^2(s_2^2s_1s_2)}+\underline{R_5s_2s_1^2s_2^2}+
		R_5s_2s_1(s_1s_2^2s_1^{-1})s_2+R_5s_2s_1^2s_2^2s_1^{-2}s_2\\
		
			&\in&\underline{R_5(s_2s_1s_2^{-1})s_1^2s_2^2}+R_5s_1^{-1}(s_1s_2s_1^{2})s_2^2s_1^{-2}s_2+U'''\\
			&\in&u_1s_2^2(s_1s_2^3s_1^{-1})s_1^{-1}s_2+U'''\\
			&\in& u_1u_2u_1s_2s_1^{-1}s_2u_1+U'''.
			\end{array}}$
		
		The result follows from proposition \ref{l2}(ii).
		\item[C4.] We will prove that $s_2u_1s_2^{-2}s_1^{2}s_2^{-2}u_1\subset U'''$. Since $s_2u_1s_2^{-2}s_1^{2}s_2^{-2}u_1=(s_2u_1s_2^{-1})s_2^{-1}s_1^{2}s_2^{-2}u_1=s_1^{-1}u_2s_1s_2^{-1}s_1^2s_2^{-2}u_1$, it will be sufficient to prove that $u_2s_1s_2^{-1}s_1^2s_2^{-2}\subset U'''$. We expand $u_2$ as $R_5+R_5s_2+R_5s_2^{-1}+R_5s_2^2+R_5s_2^3$ and we have:  $u_2s_1s_2^{-1}s_1^2s_2^{-2}\subset \underline{R_5s_1s_2^{-1}s_1^2s_2^{-2}}+\underline{R_5(s_2s_1s_2^{-1})s_1^2s_2^{-2}}+\grF(u_1s_2u_1s_2s_1^{-2}s_2^{2})+R_5s_2^2s_1s_2^{-1}s_1^2s_2^{-2}+
		R_5s_2^3s_1s_2^{-1}s_1^2s_2^{-2}$.
		By the first case (case C1) we have that $\grF(u_1s_2u_1s_2s_1^{-2}s_2^{2})
		\subset u_1\grF(U''')u_1\stackrel{\ref{r1}}{\subset} U'''$. It remains to prove that the elements $s_2^2s_1s_2^{-1}s_1^2s_2^{-2}$ and $s_2^3s_1s_2^{-1}s_1^2s_2^{-2}$ are inside $U'''$. We have:
		$s_2^2s_1s_2^{-1}s_1^2s_2^{-2}=s_2(s_2s_1s_2^{-1})s_1^2s_2^{-2}=s_2s_1^{-1}(s_2s_1^3s_2^{-1})s_2^{-1}=s_2s_1^{-2}s_2^2(s_2s_1s_2^{-1})=s_1^{-1}(s_1s_2s_1^{-1})s_1^{-1}s_2^2s_1^{-1}s_2s_1=s_1^{-1}s_2^{-1}(s_1s_2s_1^{-1})s_2^2s_1^{-1}s_2s_1=\underline{s_1^{-1}s_2^{-2}(s_1s_2^3s_1^{-1})s_2s_1}$. By using analogous calculations we have 
		$s_2^3s_1s_2^{-1}s_1^2s_2^{-2}=s_2^2(s_2s_1s_2^{-1})s_1(s_1s_2^{-2}s_1^{-1})s_1=s_2^2s_1^{-2}(s_1s_2s_1^2)s_2^{-1}s_1^{-2}s_2s_1\in s_2^2s_1^{-2}s_2^2s_1^{-1}s_2u_1$.
		We expand $s_1^{-2}$ as a linear combination  of 
		$s_1^{-1}$, $1$, $s_1$, $s_2^{2}$ and $s_2^{3}$ and we have:
		
		$\small{\begin{array}[t]{lcl}s_2^2s_1^{-2}s_2^2s_1^{-1}s_2&\in& R_5s_2^2s_1^{-1}s_2^2s_1^{-1}s_2+\underline{R_5s_2^4s_1^{-1}s_2}+
		\underline{R_5s_2(s_2s_1s_2^2)s_1^{-1}s_2}+
		R_5s_2^2s_1^{2}s_2^2s_1^{-1}s_2+\\&&+
	R_5s_2^2s_1^{3}s_2^2s_1^{-1}s_2\\
	&	\in& R_5s_2^3(s_2^{-1}s_1^{-1}s_2)s_2s_1^{-1}s_2+R_5
	s_2^2s_1(s_1s_2^{2}s_1^{-1})s_2
		+R_5s_2^2s_1^2(s_1s_2^2s_1^{-1})s_2+U'''\\
		&\in& R_5s_2^3s_1(s_2^{-1}s_1^{-1}s_2)s_1^{-1}s_2+
		R_5s_2(s_2s_1s_2^{-1})s_1^{2}s_2^{2}
		+R_5s_2(s_2s_1^2s_2^{-1})s_1^2s_2^2+U'''\\
		&\in&\underline{R_5s_2^3s_1^2(s_2^{-1}s_1^{-2}s_2)}+R_5s_2s_1^{-1}s_2s_1^{3}s_2^{2}+
		R_5s_2s_1^{-1}s_2^2s_1^3s_2^2+U'''
		\end{array}}$
	
	Therefore, it remains to prove that $B:=R_5s_2s_1^{-1}s_2s_1^{3}s_2^{2}+
	R_5s_2s_1^{-1}s_2^2s_1^3s_2^2\subset U'''$. We expand  $s_1^3$ as a linear combination  of 
	$s_1^2$, $s_1$, 1 $s_1^{-1}$ and $s_1^{-2}$ and we have that $B\subset R_5s_2s_1^{-1}s_2(R_5s_1^2+R_5s_1+R_5+R_5s_1^{-1}+R_5s_1^{-2})s_2^{2}+
	R_5s_2s_1^{-1}s_2^2(R_5s_1^2+R_5s_1+R_5+R_5s_1^{-1}+R_5s_1^{-2})s_2^2$. 
By cases C1, C2 and C3 we have: 
	
$\small{\begin{array}[t]{lcl}
			 B&\subset&
		R_5s_2s_1^{-1}\grv s_2+\underline{R_5s_2s_1^{-1}(s_2s_1s_2^2)}+\underline{R_5s_2s_1^{-1}s_2^3u_1}+
			R_5s_2s_1^{-1}s_2s_1^{-1}s_2^2+\underline{R_5s_2s_1^{-1}(s_2^2s_1s_2)s_2}+\\&&+\underline{R_5s_2s_1^{-1}s_2^4}+
			R_5s_2s_1^{-1}s_2^2s_1^{-1}s_2^2+U'''\\
			&\subset&R_5s_2\grv s_1^{-1}s_2+
			R_5s_2s_1^{-1}s_2s_1^{-1}s_2^2+
			R_5s_2s_1^{-1}s_2^2s_1^{-1}s_2^2+U'''\\
			&\subset&u_2u_1s_2s_1^{-1}s_2+R_5s_2^2(s_2^{-1}s_1^{-1}s_2)s_1^{-1}s_2^2+
			R_5s_2s_1^{-2}(s_1s_2^2s_1^{-1})s_2^2+U'''\\
			&\stackrel{\ref{l2}}{\subset}&R_5s_2^2s_1(s_2^{-1}s_1^{-2}s_2)s_2+
			R_5(s_2s_1^{-2}s_2^{-1})s_1^2s_2^3+U'''\\
			&\subset&R_5s_2^2s_1^{2}s_2^{-1}(s_2^{-1}s_1^{-1}s_2)+U''' \\&\subset& u_1u_2u_1s_2^{-1}s_1s_2^{-1}+U'''
			.
			\end{array}}$
		
		The result follows from proposition \ref{l2}(ii).
		\item[C5.] We 
	will prove that $s_2u_1s_2^2s_1^2s_2^{-2}u_1\subset U'''$. For this purpose, we use straight-forward calculations and we have
	 	$s_2u_1s_2^2s_1^2s_2^{-2}=(s_2u_1s_2^{-1})s_2^2(s_2s_1^2s_2^{-1})s_2^{-1}
		=s_1^{-1}u_2(s_1s_2^2s_1^{-1})s_2(s_2s_1s_2^{-1})=
		s_1^{-1}u_2s_1(s_1s_2^2s_1^{-1})s_2s_1=
		s_1^{-1}u_2(s_2s_1s_2^{-1})s_1^2s_2^2s_1=s_1^{-2}(s_1u_2s_1^{-1})s_2s_1^3s_2^2s_1=s_1^{-2}s_2^{-1}u_1s_2^2s_1^3s_2^2s_1$, meaning that $s_2u_1s_2^2s_1^2s_2^{-2}u_1
		\subset u_1s_2^{-1}u_1s_2^2s_1^3s_2^2u_1$.
		Hence, we have to prove that $s_2^{-1}u_1s_2^2s_1^3s_2^2\subset U'''$.
		For this purpose, we expand  $s_1^3$ as a linear combination  of 
		$s_1^2$, $s_1$, 1 $s_1^{-1}$ and $s_1^{-2}$ and we have that $s_2^{-1}u_1s_2^2s_1^3s_2^2\subset \grF(s_2V+s_2u_1s_2^{-2}s_1^2s_2^{-2})+s_2^{-1}u_1s_2^2s_1s_2^2+
		s_2^{-1}u_1s_2^2s_1^{-1}s_2^2$.
		By relation (\ref{7}) and case C4 we have that $\grF(s_2V+s_2u_1s_2^{-2}s_1^2s_2^{-2})\subset \grF(U''')\stackrel{\ref{r1}}{\subset} U'''$. Moreover, 
		$s_2^{-1}u_1s_2^2s_1s_2^2=s_2^{-1}u_1(s_2^2s_1s_2)s_2=s_2^{-1}u_1\grv=\underline{s_2^{-1}\grv u_1}.$ It remains to prove that $s_2^{-1}u_1s_2^2s_1^{-1}s_2^2\subset U'''$. We have:
		$s_2^{-1}u_1s_2^2s_1^{-1}s_2^2=(s_2^{-1}u_1s_2)s_2s_1^{-1}s_2^2=s_1u_2(s_1^{-1}s_2s_1)s_1^{-2}s_2^2=s_1u_2s_1(s_2^{-1}s_1^{-2}s_2)s_2\subset u_1u_2u_1s_2^{-1}s_1s_2^{-1}u_1$.
		The result follows from proposition \ref{l2}(i).
		\qedhere
	\end{itemize}
\end{proof}
From now on we will double-underline the elements of the form $u_1s_2^{\pm}u_1u_2u_1u_2u_1$ and  we will use the fact that they are elements of $U'''$ (proposition \ref{p2}) without mentioning it.

We can now prove the following lemma that helps us to ``replace''  inside the definition of $U''''$ the element $\grv^4$ by the element $s_2^{-2}s_1^2s_2^2s_1^3s_2^2$ modulo $U'''$.  
\begin{lem}  $s_2^{-2}s_1^2s_2^2s_1^3s_2^2\in u_1\grv^3+u_1^{\times}\grv^4+u_1s_2u_1u_2u_1u_2u_1\subset U''''.$
	\label{ll}
\end{lem}
\begin{proof} In this proof we will double-underline only the elements of the form $u_1s_2u_1u_2u_1u_2u_1$ (and not of the form $u_1s_2^{-1}u_1u_2u_1u_2u_1$ ). The fact that $u_1\grv^3+u_1^{\times}\grv^4+u_1s_2u_1u_2u_1u_2u_1$ is a subset of $ U''''$ follows from the definition of $U''''$ and proposition \ref{p2}. As a result, we restrict ourselves to proving that $s_2^{-2}s_1^2s_2^2s_1^3s_2^2\in u_1\grv^3+u_1^{\times}\grv^4+u_1s_2u_1u_2u_1u_2u_1$.
	We first notice that
	$$\small{\begin{array}{lcl}
	s_2^{-2}s_1^2s_2^2s_1^3s_2^2&=&s_1(s_1^{-1}s_2^{-2}s_1)s_2^{-2}(s_2^2s_1s_2)s_2s_1^2(s_1s_2^2s_1^{-1})s_1^{-1}s_1^2\\
	&=&
	s_1s_2s_1^{-2}s_2^{-3}s_1\grv s_1^2s_2^{-1}s_1(s_1s_2s_1^{-1})s_1^2\\
	&=&
	s_1s_2s_1^{-2}s_2^{-3}s_1^3(s_2s_1^3s_2^{-1})s_1s_2s_1^2\\
	&=&
	s_1s_2s_1^{-2}s_2^{-3}s_1^2s_2^3s_1^2s_2s_1^2\\
	&\in& u_1s_2s_1^{-2}s_2^{-3}s_1^2s_2^3s_1^2s_2u_1.
	\end{array}}$$
	We expand $s_2^{-3}$ as a linear combination of $s_2^{-2}$, $s_2^{-1}$, 1, $s_2$ and $s_2^2$, where the coefficient of $s_2^2$ is invertible, and we have:
	
	$$\small{\begin{array}[t]{lcl}
	s_2s_1^{-2}s_2^{-3}s_1^2s_2^3s_1^2s_2&\in& R_5s_2s_1^{-2}s_2^{-2}s_1^2s_2^3s_1^2s_2+R_5s_2s_1^{-2}s_2^{-1}s_1^2s_2^3s_1^2s_2+\underline{\underline{R_5s_2s_2^3s_1^2s_2}}+\\&&+R_5s_2s_1^{-2}s_2s_1^2s_2^3s_1^2s_2+u_1^{\times}s_2s_1^{-2}s_2^{2}s_1^2s_2^3s_1^2s_2u_1^{\times}.
		\end{array}}$$
However, we notice that $s_2s_1^{-2}s_2^{-2}s_1^2s_2^3s_1^2s_2=s_2s_1^{-1}(s_1^{-1}s_2^{-2}s_1)s_1s_2^3s_1^2s_2=s_2s_1^{-1}s_2^2\grv^{-1}s_1s_2^3s_1^2s_2=s_2s_1^{-1}s_2^2s_1\grv^{-1}s_2^3s_1^2s_2=s_2s_1^{-1}s_2^2s_1(s_2^{-1}s_1^{-2}s_2)\grv=s_2s_1^{-1}s_2^2s_1^2s_2^{-2}\grv s_1^{-1} =\underline{\underline{s_2s_1^{-1}s_2^2s_1^2(s_2^{-1}s_1^2s_2)}}s_1^{-1}$. Moreover, we have  $s_2s_1^{-2}s_2^{-1}s_1^2s_2^3s_1^2s_2=s_2s_1^{-2}(s_2^{-1}s_1^2s_2)s_2^2s_1^2s_2=s_2s_1^{-1}s_2^2(s_1^{-1}s_2^2s_1)s_2(s_2^{-1}s_1s_2)=\underline{\underline{s_2s_1^{-1}s_2^3(s_1^3s_2s_1)s_1^{-2}}}.$ We also have $s_2s_1^{-2}s_2s_1^2s_2^3s_1^2s_2=s_2s_1^{-3}(s_1s_2s_1^2)s_2^3s_1^2s_2=s_2s_1^{-3}s_2(s_2s_1s_2^4)s_1^2s_2\in s_2s_1^{-3}(s_2u_1s_2u_1s_2u_1).$ However, by lemma \ref{oo}(i) we have that $s_2s_1^{-3}(s_2u_1s_2u_1s_2u_1)\subset s_2s_1^{-3}(\grv^2u_1+u_1u_2u_1u_2u_1)\subset s_2\grv^2 u_1+ \underline{\underline{u_1s_2u_1u_2u_1u_2u_1}}$. By lemma \ref{oo}(ii) we  also have $s_2\grv^2 u_1\subset \underline{\underline{u_1s_2u_1u_2u_1u_2u_1}}$.

It remains to prove that $s_2s_1^{-2}s_2^{2}s_1^2s_2^3s_1^2s_2\in u_1\grv^3+u_1^{\times}\grv^4+u_1s_2u_1u_2u_1u_2u_1$. We have:
	$$\small{\begin{array}[t]{lcl}
		s_2s_1^{-2}s_2^{2}s_1^2s_2^3s_1^2s_2&=& 
		s_2(-de^{-1}s_1^{-1}-ce^{-1}-e^{-1}bs_1-e^{-1}as_1^2+e^{-1}s_1^3)s_2^{2}s_1^2s_2^3s_1^2s_2\\

		&\in&
		R_5s_2s_1^{-1}s_2^2s_1^2s_2^3s_1^2s_2+R_5s_2^3s_1^2s_2^3s_1^2s_2
		+\underline{\underline{R_5(s_2s_1s_2^2)s_1^2s_2^3s_1^2s_2u_1}}+R_5s_2s_1^2s_2^2s_1^2s_2^3s_1^2s_2+\\&&+
		u_1^{\times}s_2s_1^3s_2^2s_1^2s_2^2\grv.
		\end{array}}$$
	We first notice that we have  $s_2s_1^{-1}s_2^2s_1^2s_2^3s_1^2s_2=s_2(s_1^{-1}s_2^2s_1)s_1s_2^3s_1^2s_2=s_2^2s_1^2(s_2^{-1}s_1s_2)s_2^2s_1^2s_2=s_1(s_1^{-1}s_2^2s_1)s_1^2s_2(s_1^{-1}s_2^2s_1)(s_1s_2s_1)s_1^{-1}=s_1s_2s_1^2(s_2^{-1}s_1^2s_2)s_2s_1^3s_2s_1^{-1}=\underline{\underline{s_1s_2s_1^3s_2^3(s_2^{-1}s_1^{-1}s_2)(s_1^3s_2s_1)s_1^{-2}}}.$
	Moreover, we have that $s_2^3s_1^2s_2^3s_1^2s_2=s_1(s_1^{-1}s_2^3s_1)s_1s_2^3s_1^2s_2=s_1s_2s_1^3(s_2^{-1}s_1s_2)s_2^2s_1(s_1s_2s_1)s_1^{-1}=\underline{\underline{s_1s_2s_1^4s_2(s_1^{-1}s_2^2s_1)s_2s_1s_2s_1^{-1}}}.$ Using analogous calculations, we also have that $s_2s_1^2s_2^2s_1^2s_2^3s_1^2s_2=s_1^{-1}(s_1s_2s_1^2)s_2^2s_1^2s_2^3s_1^2s_2=
	s_1^{-1}s_2(s_2s_1s_2^3)s_1^2s_2^3s_1^2s_2=s_1^{-1}s_2s_1^2(s_1s_2s_1^3)s_2^3s_1^2s_2=s_1^{-1}s_2s_1^2s_2^2(s_2s_1s_2^4)s_1^2s_2\in u_1\grv (s_2u_1s_2u_1s_2u_1).$ However, by lemma \ref{oo}(i) we have that $u_1\grv (s_2u_1s_2u_1s_2u_1)\subset u_1\grv(\grv^2u_1+u_1u_2u_1u_2u_1)\subset u_1\grv^3+\underline{\underline{u_1\grv u_2u_1u_2u_1}}.$ 
	
	In order to finish the proof, it remains to prove that $s_2s_1^3s_2^2s_1^2s_2^2\grv\in u_1^{\times}\grv^4+u_1s_2u_1u_2u_1u_2u_1$. 
	We use lemma \ref{cc} and we have: \\\\
		$\small{\begin{array}[t]{lcl}
s_2s_1^3s_2^2s_1^2s_2^2\grv&\in&
	u_1^{\times}(u_1s_2u_1s_2s_1^3s_2+u_1s_2^2s_1^3s_2s_1^{-1}s_2+
		u_1u_2u_1u_2u_1+u_1^{\times}\grv^3)\grv\\
		&\in&
		u_1s_2u_1s_2s_1^4(s_1^{-1}s_2^2s_1)s_1s_2+u_1s_2^2s_1^3s_2(s_1^{-1}s_2s_1)s_1^{-1}\grv +u_1^{\times}\grv^4\\

		&\in&
		u_1s_2u_1s_2s_1^4s_2s_1^2(s_2^{-1}s_1s_2)+u_1s_2^2s_1^3s_2^2s_1s_2^{-1}\grv +u_1^{\times}\grv^4\\

		&\in&
		u_1s_2u_1s_2s_1^4s_2s_1^3s_2u_1+u_1s_2^2s_1^3s_2^2s_1^3s_2+u_1^{\times}\grv^4\\
		
		&\in&
		u_1s_2u_1(s_2u_1s_2u_1s_2u_1)+u_1(s_1^{-1}s_2^2s_1)s_1^2s_2^2s_1^3s_2+
		u_1^{\times}\grv^4\\
		&\stackrel{\ref{oo}(i)}{\in}& u_1s_2u_1(\grv^2u_1+u_1u_2u_1u_2u_1)+
		u_1s_2s_1^2(s_2^{-1}s_1^2s_2)s_2s_1^3s_2+u_1^{\times}\grv^4\\

		&\in&
		u_1s_2\grv^2u_1+u_1s_2u_1u_2u_1u_2u_1+ u_1s_2s_1^3s_2^2(s_1^{-1}s_2s_1)s_1^2s_2+u_1^{\times}\grv^4\\
		&\stackrel{\ref{oo}(ii)}{\in}&u_1s_2u_1u_2u_1u_2u_1+\underline{\underline{u_1s_2s_1^3s_2^3s_1(s_2^{-1}s_1^2s_2)}}+u_1^{\times}\grv^4
		.
		\end{array}}$\\

\end{proof}
\begin{prop}\mbox{}  
	\vspace*{-\parsep}
	\vspace*{-\baselineskip}\\ 
	\begin{itemize}[leftmargin=0.6cm]
		\item [(i)]$s_2u_1u_2u_1s_2s_1^{-1}s_2\subset U''''.$
		\item[(ii)]$s_2u_1u_2u_1s_2^{-1}s_1s_2^{-1}\subset U''''.$
		\item[(iii)]$s_2u_1u_2u_1u_2\grv \subset U''''$.
	\end{itemize}
	\label{xx}
\end{prop}
\begin{proof}\mbox{} 
	\vspace*{-\parsep}
	\vspace*{-\baselineskip}\\
	\begin{itemize}[leftmargin=0.6cm]

	\item[(i)] By proposition \ref{l2}(ii) we have  $	s_2u_1u_2u_1s_2s_1^{-1}s_2\subset s_2u_1(u_1\grv^{2}+R_5s_2^2s_1^{-2}s_2s_1^{-1}s_2+u_1u_2u_1u_2u_1)$ and, hence, by lemma \ref{oo}(ii) we have  $s_2u_1u_2u_1s_2s_1^{-1}s_2\subset s_2u_1s_2^2s_1^{-2}s_2s_1^{-1}s_2+\underline{\underline{s_2u_1u_2u_1u_2u_1}}+U''''.$ As a result, we must prove that $s_2u_1s_2^2s_1^{-2}s_2s_1^{-1}s_2\subset U''''$. For this purpose, we expand $u_1$ as $R_5+R_5s_1+R_5s_1^{-1}+R_5s_1^2+R_5s_1^3$ and we have:
	
	$\begin{array}{lcl}		s_2u_1s_2^2s_1^{-2}s_2s_1^{-1}s_2&\subset& u_2u_1s_2s_1^{-1}s_2+\underline{\underline{
			R_5(s_2s_1s_2^2)s_1^{-2}s_2s_1^{-1}s_2}}+R_5s_2s_1^{-1}s_2^2s_1^{-2}s_2s_1^{-1}s_2
	+\\&&+R_5s_2s_1^2s_2^2s_1^{-2}s_2s_1^{-1}s_2+R_5s_2s_1^3s_2^2s_1^{-2}s_2s_1^{-1}s_2.
	\end{array}$
	
	By proposition \ref{l2}(ii) we have that $u_2u_1s_2s_1^{-1}s_2\subset U''''$. Moreover, $s_2s_1^2s_2^2s_1^{-2}s_2s_1^{-1}s_2=s_1^{-1}(s_1s_2s_1^2)s_2^2s_1^{-3}(s_1s_2s_1^{-1})s_2=s_1^{-1}s_2(s_2s_1s_2^3)s_1^{-3}s_2^{-1}s_1s_2^2=
	\underline{\underline{s_1^{-1}s_2s_1^3(s_2s_1^{-2}s_2^{-1})s_1s_2^2}}$. We also notice that $s_2s_1^3s_2^2s_1^{-2}s_2s_1^{-1}s_2=s_1^{-1}(s_1s_2s_1^3)s_2^2s_1^{-3}(s_1s_2s_1^{-1})s_2=s_1^{-1}s_2^3(s_1s_2^3s_1^{-1})s_1^{-2}s_2^{-1}s_1s_2^2=s_1^{-1}s_2^2s_1^3(s_2s_1^{-2}s_2^{-1})s_1s_2^2=s_1^{-2}(s_1s_2s_1^{-1})s_1(s_2s_1^2s_2^{-1})s_2^{-1}s_1^2s_2^2=s_1^{-2}s_2^{-1}(s_1s_2^3s_1^{-1})s_2^{-1}(s_2s_1^2s_2^{-1})s_1^2s_2^2
	\in u_1s_2^{-2}s_1^2s_2^2s_1^3s_2^2\stackrel{\ref{ll}}{\subset} U''''.$
	
		It remains to prove that the element $s_2s_1^{-1}s_2^2s_1^{-2}s_2s_1^{-1}s_2$ is inside $U''''$.
		We expand $s_2^2$ as a linear combination of $s_2$, 1, $s_2^{-1}$,  $s_2^{-2}$ and  $s_2^{-3}$ and we have
		
		$\small{\begin{array}[t]{lcl}
		s_2s_1^{-1}s_2^2s_1^{-2}s_2s_1^{-1}s_2&\in &	s_2s_1^{-1}(R_5s_2+R_5+R_5s_2^{-1}+R_5s_2^{-2}+R_5s_2^{-3})s_1^{-2}s_2s_1^{-1}s_2\\

			&\in&R_5s_2s_1^{-2}(s_1s_2s_1^{-1})s_1^{-1}(s_2s_1^{-1}s_2^{-1})s_2^2+
			\underline{\underline{R_5s_2s_1^{-3}s_2s_1^{-1}s_2}}+\\&&+\underline{\underline{
					R_5(s_2s_1^{-1}s_2^{-1})s_1^{-2}s_2s_1^{-1}s_2}}+\underline{\underline{
					R_5(s_2s_1^{-1}s_2^{-1})(s_2^{-1}s_1^{-2}s_2)s_1^{-1}s_2}}+\\&&+R_5(s_2s_1^{-1}s_2^{-1})s_2^{-1}(s_2^{-1}s_1^{-2}s_2)s_1^{-1}s_2\\

			&\in&
			\underline{\underline{R_5s_2s_1^{-2}s_2^{-1}s_1(s_2s_1^{-2}s_2^{-1})s_1s_2^2}}+
			R_5s_1^{-1}s_2^{-1}s_1^2(s_1^{-1}s_2^{-1}s_1)s_2^{-2}s_1^{-2}s_2
			+U''''\\
			&\in&R_5s_1^{-1}s_2^{-1}s_1^2s_2(s_1^{-1}s_2^{-3}s_1)s_1^{-3}s_2+
		U''''\\

			&\in&\underline{\underline{R_5s_1^{-1}s_2^{-1}s_1^2s_2^2s_1^{-3}(s_2^{-1}s_1^{-3}s_2)}}+U''''.
			\end{array}}$
		\item[(ii)]By proposition \ref{l2} (i)  we have that $s_2u_1(u_2u_1s_2^{-1}s_1s_2^{-1})\subset s_2u_1(u_1\grv^{-2}+R_5s_2^{-2}s_1^{2}s_2^{-1}s_1s_2^{-1}+u_1u_2u_1u_2u_1)
		\subset\underline{\underline{s_2\grv^{-2}u_1}}+
		s_2u_1s_2^2s_1^{-2}s_2s_1^{-1}s_2+\underline{\underline{s_2u_1u_2u_1u_2u_1}}$. Therefore, it remains to prove that  $s_2u_1s_2^2s_1^{-2}s_2s_1^{-1}s_2 \subset U''''$. We expand $u_1$ as $R_5+R_5s_1+R_5s_1^{-1}+R_5s_1^2+R_5s_1^{3}$ and we have:
		
		$\small{\begin{array}[t]{lcl}

		s_2u_1s_2^2s_1^{-2}s_2s_1^{-1}s_2	&\subset&
			\underline{\underline{R_5s_2^{-1}s_1^2s_2^{-1}s_1s_2^{-1}}}+\underline{\underline{R_5s_1^{-2}(s_1^2s_2s_1)s_2^{-2}s_1^2s_2^{-1}s_1s_2^{-1}}}+\\&&
			+R_5(s_2s_1^{-1}s_2^{-1})(s_2^{-1}s_1^2s_2)s_2^{-2}s_1s_2^{-1}+R_5(s_2s_1^2s_2^{-1})s_2^{-2}(s_2s_1^2s_2^{-1})s_1s_2^{-1}+\\&&+R_5(s_2s_1^3s_2^{-1})s_2^{-1}s_1^2s_2^{-1}(s_1s_2s_1^{-1})s_1\\
			&\subset&\underline{\underline{R_5s_1^{-1}s_2^{-1}s_1^2s_2^2(s_1^{-1}s_2^{-2}s_1)s_2^{-1}}}+
			\underline{\underline{R_5s_1^{-1}s_2^2(s_1s_2^{-2}s_1^{-1})s_2(s_2s_1s_2^{-1})}}+\\
			&&+R_5s_1^{-1}s_2^2(s_2s_1s_2^{-1})s_1(s_1s_2^{-2}s_1^{-1})s_2s_1+U''''\\
			&\subset&R_5s_1^{-1}s_2^2s_1^{-1}(s_2s_1^{2}s_2^{-1})s_1^{-2}s_2^2s_1+U''''\\
			&\subset&R_5s_1^{-1}s_2^2s_1^{-3}(s_1s_2^2s_1^{-1})s_2^2s_1+U''''\\
			&\subset& \underline{\underline{R_5s_1^{-1}s_2(s_2s_1^{-3}s_2^{-1})s_1^2s_2^3s_1}}+U''''.
			\end{array}}$
		\item[(iii)] We notice that $s_2u_1u_2u_1u_2\grv=s_2u_1u_2u_1u_2s_1^2s_2
		s_2u_1u_2u_1u_2(s_2^{-1}s_1^2s_2)
	\subset s_2u_1u_2u_1u_2s_1s_2^2u_1$. We expand $\bold{u_2}$ as $R_5+R_5s_2+R_5s_2^{-1}+R_5s_2^{2}+R_5s_2^{-2}$ and we have:

		$\small{\begin{array}[t]{lcl}

			s_2u_1u_2u_1\bold{u_2}s_1s_2^2u_1&\subset&\underline{\underline{s_2u_1u_2u_1s_2^2u_1}}+\underline{\underline{s_2u_1u_2u_1(s_2s_1s_2^2)u_1}}+
			s_2u_1u_2u_1(s_2^{-1}s_1s_2)s_2u_1+\\&&+
			s_2u_1u_2u_1s_2(s_2s_1s_2^2)u_1+s_2u_1u_2u_1s_2^{-1}(s_2^{-1}s_1s_2)s_2u_1\\
			&\subset&s_2u_1(u_2u_1s_2s_1^{-1}s_2u_1)+\underline{\underline{s_2u_1u_2u_1\grv u_1}}+s_2u_1u_2u_1(s_2^{-1}s_1s_2)s_1^{-1}s_2u_1+U''''\\
			&\stackrel{(i)}{\subset}&s_2u_1u_2u_1s_2s_1^{-2}s_2u_1+U''''.
			\end{array}}$
		
	It remains to prove that $s_2u_1u_2u_1s_2s_1^{-2}s_2u_1\subset U''''$. For this purpose, we expand $s_1^{-2}$ as a linear combination of $s_1^{-1}$, 1, $s_1$, $s_1^2$, $s_1^3$ and we have: $s_2u_1u_2u_1s_2s_1^{-2}s_2u_1	\subset s_2u_1u_2u_1s_2s_1^{-1}s_2u_1+\underline{\underline{s_2u_1u_2u_1s_2^2u_1}}+
	\underline{\underline{s_2u_1u_2u_1(s_2s_1s_2)u_1}}+s_2u_1u_2u_1\grv u_1+s_2u_1u_2u_1(s_1^{-1}s_2s_1)s_1^2s_2u_1\stackrel{(i)}{\subset}
	\underline{\underline{s_2u_1u_2\grv u_1}}+s_2u_1u_2u_1s_2s_1(s_2^{-1}s_1^2s_2)u_1+U''''
	\subset s_2u_1u_2u_1\grv s_2u_1+U''''
	\subset s_2u_1u_2\grv u_1s_2u_1+U''''$.
		
		However, 
		$\small{\begin{array}[t]{lcl}
			s_2u_1u_2\grv u_1s_2u_1&\subset&s_2u_1u_2s_1^2s_2(R_5+R_5s_1+
			R_5s_1^{-1}+R_5s_1^2+R_5s_1^3)s_2u_1\\
			&\subset&\underline{\underline{s_2u_1u_2s_1^2s_2^2u_1}}+\underline{\underline{s_2u_1u_2(s_1^2s_2s_1)s_2u_1}}
			+s_2u_1(u_2u_1s_2s_1^{-1}s_2u_1)+\\&&+
			s_2u_1u_2s_1^2\grv u_1+s_2u_1(s_1^{-1}u_2s_1)(s_1s_2s_1^3)s_2u_1\\
			&\stackrel{(i)}{\subset}&\underline{\underline{s_2u_1u_2\grv u_1}}+s_2u_1s_2u_1(s_2^2s_1s_2)s_2u_1+U''''\\
			&\subset&s_2u_1s_2u_1\grv u_1+U''''.
			\end{array}}$
		
		The result follows from the fact that $s_2u_1s_2u_1\grv u_1=\underline{\underline{s_2u_1s_2\grv u_1}}$.
		\qedhere
	\end{itemize}
\end{proof}
We can now prove the following lemma that helps us to ``replace''  inside the definition of $U$ the elements $\grv^5$ and $\grv^{-5}$ by the elements $s_2^{-2}s_1^2s_2^3s_1^2s_2^3$ and $s_2^{-2}s_1^{2}s_2^{-2}s_1^2s_2^{-2}$ modulo $U''''$, respectively.
\begin{lem}\mbox{}  
	\vspace*{-\parsep}
	\vspace*{-\baselineskip}\\ 
	\begin{itemize}[leftmargin=0.6cm]
		\item [(i)]$s_2^{-2}s_1^2s_2^3s_1^2s_2^3\in u_1^{\times}\grv^5+ U''''.$
		\item [(ii)]$s_2^{-2}s_1^{2}s_2^{-2}s_1^2s_2^{-2}\in u_1^{\times}\grv^{-5}+U''''.$
		\item  [(iii)]$s_2^{-2}u_1s_2^{-2}s_1^{2}s_2^{-2}\subset U.$
	\end{itemize}
	\label{lol}
\end{lem}
\begin{proof}\mbox{}  
	\vspace*{-\parsep}
	\vspace*{-\baselineskip}\\ 
	\begin{itemize}[leftmargin=0.6cm]
		\item[(i)] We notice that $s_2^{-2}s_1^2s_2^3s_1^2s_2^3=s_2^{-2}s_1^2s_2^3s_1(s_1s_2^3s_1^{-1})s_1
		=s_2^{-2}s_1^2s_2^2(s_2s_1s_2^{-1})s_1^2(s_1s_2s_1^{-1})s_1^2		=s_2^{-2}s_1^2s_2^2s_1^{-1}(s_2s_1^3s_2^{-1})s_1s_2s_1^2
		=s_2^{-2}s_1^2s_2^2s_1^{-2}s_2^2\grv s_1^2$. 
		We expand $s_1^{-2}$ as a linear combination of $s_1^{-1}$, 1, $s_1$ $s_1^{2}$ and $s_1^3$, where the coefficient of $s_1^3$ is invertible, and we have:
		
		$\small{\begin{array}[t]{lcl}
		s_2^{-2}s_1^2s_2^2s_1^{-2}s_2^2\grv s_1^2	
			&\in&s_2^{-2}s_1(s_1s_2^2s_1^{-1})s_2^2\grv u_1+ s_2^{-1}(s_2^{-1}s_1^4s_2)s_2^2\grv u_1+\\&&+s_2^{-2}s_1^2(s_2^2s_1s_2)s_2\grv u_1+s_2^{-1}(s_2^{-1}s_1^2s_2)s_2s_1^2s_2^2\grv s_2\grv u_1+
			u_1^{\times}s_2^{-2}s_1^2s_2^2s_1^{3}s_2^2\grv s_1^2\\
			&\stackrel{\ref{ll}}{\in}&\underbrace{s_1(s_1^{-1}s_2^{-2}s_1)s_2^{-1}s_1^2s_2^3\grv u_1+s_1(s_1^{-1}s_2^{-1}s_1)s_2^{4}s_1^{-1}s_2^2 s_1\grv u_1}_{\in u_1s_2u_1u_2u_1u_2\grv u_1}
			+s_2^{-2}s_1^3\grv^2u_1+\\&&+
			(s_2^{-1}s_1s_2)s_2(s_1^{-1}s_2s_1)s_1s_2^2\grv u_1+u_1^{\times}(u_1\grv^3+u_1^{\times}\grv^4+u_1s_2u_1u_2u_1u_2u_1)\grv s_1^2\\
			&\in&\underline{\underline{s_2^{-2}\grv^2u_1}}+s_1s_2s_1^{-1}s_2^2s_1(s_2^{-1}s_1s_2)s_2\grv  u_1+u_1\grv^3+u_1^{\times}\grv^5+u_1s_2u_1u_2u_1u_2u_1\grv u_1\\
			&\in&s_1s_2s_1^{-1}(u_2s_1^2s_2s_1^{-1}s_2)\grv u_1+u_1^{\times}\grv^5+u_1s_2u_1u_2u_1u_2u_1\grv u_1+U''''.
			\end{array}}$
		
		By proposition \ref{l2}(ii) we have that $s_2s_1^{-1}(u_2s_1^2s_2s_1^{-1}s_2)\grv u_1\subset s_2s_1^{-1}(u_1\grv^2+R_5s_2^2s_1^{-2}s_2s_1^{-1}s_2+u_1u_2u_1u_2u_1)\grv u_1	\stackrel{\ref{ll}}{\in}s_2\grv^2u_1+u_1s_2u_1u_2u_1u_2u_1\grv u_1$
		and, hence, the element $s_2^{-2}s_1^2s_2^2s_1^{-2}s_2^2\grv s_1^2 $ is inside $s_2\grv^2u_1+u_1^{\times}\grv^5+u_1s_2u_1u_2u_1u_2u_1\grv u_1+U''''$. We  notice that $u_1s_2u_1u_2u_1u_2u_1\grv u_1=u_1s_2u_1u_2u_1u_2\grv u_1$ and, hence, by lemma \ref{oo}(ii) and proposition \ref{xx}(iii)   we have that the element $s_2^{-2}s_1^2s_2^2s_1^{-2}s_2^2\grv s_1^2$ is inside $u_1^{\times}\grv^5+U''''$.

		\item[
		(ii)]$\small{\begin{array}[t]{lcl}
			s_2^{-2}s_1^{2}s_2^{-2}s_1^2s_2^{-2}&=&s_2^{-2}(as_1+b+cs_1^{-1}+ds_1^{-2}+es_1^{-3})s_2^{-2}s_1^2s_2^{-2}
			\\
			&\in&\underline{\underline{s_1(s_1^{-1}s_2^{-2}s_1)s_2^{-2}s_1^2s_2^{-2}}}+
			\underline{R_5s_2^{-4}s_1^2s_2}+\underline{\underline{R_5s_2^{-1}(s_2^{-1}s_1^{-1}s_2^{-2})s_1^2s_2^{-2}}}+\\&&+R_5s_2^{-1}(s_2^{-1}s_1^{-2}s_2)s_2^{-3}s_1^2s_2^{-2}+R_5
			s_2^{-2}s_1^{-2}(s_1^{-1}s_2^{-2}s_1)s_1s_2^{-2}\\
			&\in&R_5s_2^{-1}s_1s_2^{-2}(s_1^{-1}s_2^{-3}s_1)s_1s_2^{-2}+R_5
			s_2^{-1}(s_2^{-1}s_1^{-2}s_2)s_1^{-2}(s_2^{-1}s_1s_2)s_2^{-3}+U''''\\ 
			&\in&R_5s_2^{-1}s_1^2(s_1^{-1}s_2^{-1}s_1^{-3})s_2^{-1}s_1s_2^{-2}+R_5s_2^{-1}s_1s_2^{-1}(s_2^{-1}s_1^{-2}s_2)s_1^{-1}s_2^{-3}+U''''\\
			&\in&\underline{\underline{R_5s_2^{-1}s_1^2s_2^{-2}(s_2^{-1}s_1^{-1}s_2^{-2})s_1s_2^{-2}}}+R_5s_2^{-1}s_1(s_2^{-1}s_1s_2)s_2^{-2}s_1^{-2}s_2^{-3}+U''''
			\\
			&\in&R_5(s_2^{-1}s_1^2s_2)s_1^{-1}s_2^{-2}s_1^{-2}s_2^{-3}+U''''\\
			&\in& \grF(u_1s_2^{-2}s_1^2s_2^3s_1^2s_2^3)+U''''.
			\end{array}}$
		
		The result then follows from (i) and Lemma \ref{r1}.
		\item[(iii)] We expand $u_1$ as $R_5+R_5s_1+R_5s_1^{-1}+R_5s_1^2+R_5s_1^{-2}$ and we have that
		
		$\small{\begin{array}{lcl}s_2^{-2}u_1s_2^{-2}s_1^2s_2^{-2}&\subset& \underline{R_5s_2^{-4}s_1^2s_2^{-2}}+\underline{\underline{R_5s_2^{-1}(s_2^{-1}s_1s_2)s_2^{-3}s_1^2s_2^{-2}}}+\underline{\underline{R_5(s_2^{-2}s_1^{-1}s_2^{-1})s_2^{-1}s_1^2s_2^{-2}}}+\\&&+R_5s_2^{-2}s_1^2s_2^{-2}s_1^2s_2^{-2}+R_5s_2^{-2}s_1^{-2}s_2^{-2}s_1^2s_2^{-2}.
			\end{array}}$
		
	The element $s_2^{-2}s_1^2s_2^{-2}s_1^2s_2^{-2}$ is inside $U''''$, by (ii). Moreover, we have $s_2^{-2}s_1^{-2}s_2^{-2}s_1^2s_2^{-2}=
		 s_2^{-1}\grv^{-1}(s_2^{-1}s_1^2s_2)s_2^{-3}
		=s_2^{-1}\grv^{-1}s_1s_2^2s_1^{-1}s_2^{-3}
			=s_2^{-1}s_1\grv^{-1}s_2^2s_1^{-1}s_2^{-3}=\underline{\underline{s_2^{-1}s_1^2(s_1^{-1}s_2^{-1}s_1^{-2})s_2s_1^{-1}s_2^{-3}}}$.
			\qedhere
	\end{itemize}
\end{proof}
We can now prove the main theorem of this section.
\begin{thm}
	\mbox{}  
	\vspace*{-\parsep}
	\vspace*{-\baselineskip}\\ 
	\begin{itemize}[leftmargin=0.6cm]
		\item [(i)]	$U=U''''+u_1\grv^{-5}$.
		\item[(ii)]$H_5=U$.
		\label{thh2}
	\end{itemize}
\end{thm}
\begin{proof}
	\mbox{}  
	\vspace*{-\parsep}
	\vspace*{-\baselineskip}\\ 
	\begin{itemize}[leftmargin=0.6cm]
		\item [(i)]	
		By definition, $U=U''''+u_1\grv^5+u_1\grv^{-5}$. Hence, it is enough to prove that $\grv^{-5}\in u_1^{\times}\grv^{5}+U''''.$
		By lemma \ref{lol}(ii) we have that $\grv^{-5}\in u_1^{\times}s_2^{-2}s_1^{2}s_2^{-2}s_1^2s_2^{-2}+U''''$. We expand $\bold
		{s_2^{-2}}$ as a linear combination of $s_2^{-1}$, 1, $s_2$, $s_2^2$ and $s_2^3$, where the coefficient of $s_2^3$ is invertible, and we have: 
		
		$\small{\begin{array}{lcl}
			s_2^{-2}s_1^{2} s_2^{-2}s_1^2\bold {s_2^{-2}}&\in& R_5s_2^{-1}(s_2^{-1}s_1^2s_2)s_2^{-3}s_1^2s_2^{-1}+\underline{R_5s_2^{-2}s_1^2s_2^{-2}s_1^2}+
		R_5(s_1^{-1}s_2^{-2}s_1)s_1s_2^{-1}(s_2^{-1}s_1^2s_2)+\\&&+R_5s_1(s_1^{-1}s_2^{-2}s_1)s_1s_2^{-1}(s_2^{-1}s_1^2s_2)s_2+R_5^{\times}s_2^{-2}s_1^{2}s_2^{-2}s_1^2s_2^3\\
			&\in&u_1s_2^{-1}s_1s_2^2(s_1^{-1}s_2^3s_1)s_1s_2^{-1}+u_1s_2s_1^{-1}(s_1^{-1}s_2^{-1}s_1)(s_2^{-1}s_1s_2)s_2s_1^{-1}+\\&&+u_1s_2s_1^{-1}(s_1^{-1}s_2^{-1}s_1)(s_2^{-1}s_1s_2)s_2s_1^{-1}s_2+
			u_1^{\times}s_2^{-2}s_1^{2}s_2^{-2}s_1^2s_2^3+U''''\\
			&\in&u_1\grF(s_2s_1^{-1}u_2u_1s_2s_1^{-1}s_2)+
			\underline{\underline{u_1s_2s_1^{-1}s_2(s_1^{-1}s_2^{-1}s_1)s_2s_1^{-1}s_2
					s_1^{-1}}}+\\&&+u_1s_2s_1^{-1}s_2(s_1^{-1}s_2^{-1}s_1)s_2s_1^{-1}s_2
			s_1^{-1}s_2+u_1^{\times}s_2^{-2}s_1^{2}s_2^{-2}s_1^2s_2^3+U''''\\
			&\stackrel{\ref{l2}}{\in}&u_1\grF\big(s_2s_1^{-1}\grv^2u_1+s_2s_1^{-1}s_2^2s_1^{-2}s_2s_1^{-1}s_2+
			s_2s_1^{-1}u_1u_2u_1u_2u_1\big)+\\&&+
			u_1s_2s_1^{-1}s_2^2s_1^{-2}s_2s_1^{-1}s_2+u_1^{\times}s_2^{-2}s_1^{2}s_2^{-2}s_1^2s_2^3+U''''
		\end{array}}$
	
	However, by lemma \ref{oo}(ii) proposition \ref{xx}(i)  we have that $\grF\big(s_2s_1^{-1}\grv^2u_1+s_2s_1^{-1}s_2^2s_1^{-2}s_2s_1^{-1}s_2+\underline{\underline{s_2s_1^{-1}u_1u_2u_1u_2u_1}}\big)\subset \grF(U'''')\stackrel{\ref{r1}}{\subset}U''''$. Therefore, it will be sufficient to prove that the element $s_2^{-2}s_1^{2}s_2^{-2}s_1^2s_2^3$ is inside  $u_1^{\times}\grv^5+U''''$. We expand $\bold
	{s_2^{-2}}$ as a linear combination of $s_2^{-1}$, 1, $s_2$, $s_2^2$ and $s_2^3$, where the coefficient of $s_2^3$ is invertible, and we have: 
		$s_2^{-2}s_1^{2}\bold {s_2^{-2}}s_1^2s_2^3 \in u_1s_2^{-3}(s_2s_1^2s_2^{-1})s_1^2s_2^3+\underline{u_1s_2^{-2}s_1^4s_2^3}+
		\underline{\underline{u_1s_2^{-2}(s_1^2s_2s_1)s_1s_2^3}}+u_1s_2^{-1}(s_2^{-1}s_1^2s_2)s_2s_1^2s_2^3+u_1^{\times}s_2^{-2}s_1^2s_2^3s_1^2s_2^3$.
		By lemma \ref{lol}(i) we have that $u_1^{\times}s_2^{-2}s_1^2s_2^3s_1^2s_2^3\subset u_1^{\times}\grv^5+U''''.$
		Therefore, $s_2^{-2}s_1^{2}s_2^{-2}s_1^2s_2^3 \in\underline{\underline{u_1(s_1s_2^{-3}s_1^{-1})s_2^2s_1^3s_2^3}}+u_1s_2^{-1}s_1s_2^2s_1^{-1}s_2s_1^2s_2^3+u_1^{\times}\grv^5+U''''$. It remains to prove that the element $s_2^{-1}s_1s_2^2s_1^{-1}s_2s_1^2s_2^3$ is inside $U''''$. Indeed, 
		$s_2^{-1}s_1s_2^2s_1^{-1}s_2s_1^2s_2^3=(s_2^{-1}s_1s_2)s_2(s_1^{-1}s_2s_1)s_1s_2^3=s_1s_2(s_1^{-1}s_2^2s_1)(s_2^{-1}s_1s_2)s_2^2=s_1s_2^2s_1^2(s_2^{-1}s_1s_2)s_1^{-1}s_2^2=s_1s_2^3(s_2^{-1}s_1^3s_2)s_1^{-2}s_2^2=\underline{\underline{s_1^2(s_1^{-1}s_2^3s_1)s_2^3s_1^{-3}s_2^2}}$.

		\item[(ii)] As we explained in the beginning of this section, since $1\in U$ it will be sufficient to prove that $U$ is invariant under left multiplication by $s_2$. We use the fact that $U$ is equal to the RHS of (i) and by the definition of $U''''$ we have:
		
		$\small{\begin{array}{lcl} U&=&U'+\sum\limits_{k=2}^4\grv^{k}u_1+\sum\limits_{k=2}^5\grv^{-k}u_1+
		u_1s_2^{-2}s_1^2s_2^{-1}s_1s_2^{-1}u_1+u_1s_2^{2}s_1^{-2}s_2s_1^{-1}s_2u_1+
		u_1s_2s_1^{-2}s_2^{2}s_1^{-2}s_2^{2}u_1+\\&&+u_1s_2^{-1}s_1^2s_2^{-2}s_1^2s_2^{-2}u_1.
		\end{array}}$
	
		On one hand, $s_2(U'+\grv^2u_1+
		u_1s_2^{-2}s_1^2s_2^{-1}s_1s_2^{-1}u_1+u_1s_2^{2}s_1^{-2}s_2s_1^{-1}s_2u_1)\subset U$ (proposition \ref{p2}, lemma \ref{oo}$(ii)$ and proposition \ref{xx}$(i),(ii)$). 
		On the other hand, $$\sum\limits_{k=2}^{5}s_2\grv^{-k}u_1=\sum\limits_{k=2}^{5}s_1^{-2}s_2^{-1}\grv^{-k+1}u_1\subset \grF(\sum\limits_{k=2}^{5}u_1s_2\grv^{k-1}u_1)\stackrel{\ref{p1}\text{ and }\ref{oo}(ii)}{\subset}u_1\grF(U+s_2\grv^{3}+s_2\grv^4)u_1
		$$
		Therefore, by lemma \ref{r1} we only need to prove that 
		$$s_2(\grv^3u_1+\grv^4 u_1+
		u_1s_2s_1^{-2}s_2^{2}s_1^{-2}s_2^{2}u_1+
		u_1s_2^{-1}s_1^2s_2^{-2}s_1^2s_2^{-2}u_1)\subset U.$$
		We first notice that $s_2\grv^4u_1=s_2\grv^3\grv u_1=s_2\grv^3u_1\grv$. Therefore, in order to prove that $s_2(\grv^3u_1+\grv^4u_1)\subset U$, it will be sufficient to prove that $s_2\grv^3u_1\subset u_1s_2u_1u_2u_1u_2u_1$ (propositions \ref{p2} and \ref{xx}(iii)).
		Indeed, we have: $s_2\grv^3 u_1=s_2\grv^2\grv u_1
		\stackrel{\ref{oo}(ii)}{=}s_1s_2s_1^4s_2s_1^3s_2\grv u_1
		=s_1s_2s_1^4s_2s_1^4(s_1^{-1}s_2^2s_1)s_1s_2u_1
		=s_1s_2s_1^4s_2s_1^4s_2s_1^2(s_2^{-1}s_1s_2)u_1
		\subset u_1s_2u_1(s_2u_1s_2u_1s_2u_1)$.
		However, by lemma \ref{oo}(i)  we have that $u_1s_2u_1(s_2u_1s_2u_1s_2u_1)\subset u_1s_2u_1(\grv^2u_1+u_1u_2u_1u_2u_1)$.
		The result follows from lemma \ref{oo}(ii).

		It remains to prove that $s_2
		u_1s_2^{-1}s_1^2s_2^{-2}s_1^2s_2^{-2}u_1$ and $s_2u_1s_2s_1^{-2}s_2^{2}s_1^{-2}s_2^{2}u_1$ are subsets of  $U.$ We have: 
		
			$\small{\begin{array}{lcl}
				s_2u_1s_2^{-1}s_1^2s_2^{-2}s_1^2s_2^{-2}&=&s_2(R_5+R_5s_1+R_5s_1^{-1}+
				R_5s_1^2+R_5s_1^{-2})s_2^{-1}s_1^2s_2^{-2}s_1^2s_2^{-2}\\
				&\subset&\underline{R_5s_1^2s_2^{-2}s_1^2s_2^{-2}}+\underline{\underline{
						R_5(s_2s_1s_2^{-1})s_1^2s_2^{-2}s_1^2s_2^{-2}}}+\underline{\underline{
						R_5(s_2s_1^{-1}s_2^{-1})s_1^2s_2^{-2}s_1^2s_2^{-2}}}+\\&&
				+R_5(s_2s_1^{2}s_2^{-1})s_1^2s_2^{-2}s_1^2s_2^{-2}+
				R_5(s_2s_1^{-2}s_2^{-1})s_1^2s_2^{-2}s_1^2s_2^{-2}
			\\
				&\subset&u_1s_2^2s_1^3s_2^{-2}s_1^2s_2^{-2}+u_1s_2^{-2}u_1s_2^{-2}s_1^2s_2^{-2}+U.
				\end{array}}$
		
	However, by lemma \ref{lol}(iii) we have that $u_1s_2^{-2}u_1s_2^{-2}s_1^2s_2^{-2}\subset U$. Therefore, it remains to prove that the element $s_2^2s_1^3s_2^{-2}s_1^2s_2^{-2}$ is inside $ U$. For this purpose, we expand 
	$s_1^3$ as a linear combination of $s_1^2$, $s_1$, 1, $s_1^{-1}$ and $s_1^{-2}$ and  we have:
	
	$\small{\begin{array}{lcl}s_2^2s_1^3s_2^{-2}s_1^2s_2^{-2}&\in&
	R_5s_2^2s_1^2s_2^{-2}s_1^2s_2^{-2}+\underline{\underline{
			R_5s_2^2(s_1s_2^{-2}s_1^{-1})s_1^3s_2^{-2}}}+\underline{u_1s_2^{-2}}+
	\underline{\underline{R_5s_1^{-1}(s_1s_2^2s_1^{-1})s_2^{-2}s_1^2s_2^{-2}}}+\\&&+
	R_5s_2^{2}s_1^{-2}s_2^{-2}s_1^{2}s_2^{-2}.
	\end{array}}$
	
	However, $s_2^2s_1^2s_2^{-2}s_1^2s_2^{-2}=s_2(s_2s_1^2s_2^{-1})s_2^{-1}s_1(s_1s_2^{-2}s_1^{-1})s_1=s_2s_1^{-1}s_2(s_2s_1s_2^{-1})s_1s_2^{-1}s_1^{-2}s_2s_1=s_2s_1^{-2}(s_1s_2s_1^{-1})(s_2s_1^2s_2^{-1})s_1^{-2}s_2s_1=s_2s_1^{-2}s_2^{-1}(s_1s_2s_1^{-1})s_2^2s_1^{-1}s_2s_1=\underline{\underline{u_1s_2s_1^{-2}s_2^{-2}(s_1s_2^3s_1^{-1})s_2s_1}}$. 
	Moreover, we expand $s_1^2$ as a linear combination of $s_1$, 1, $s_1^{-1}$, $s_1^{-2}$ and $s_1^{-3}$ and we have:

		$\small{\begin{array}{lcl}

		s_2^{2}s_1^{-2}s_2^{-2}s_1^{2}s_2^{-2}

			&\in &\underline{R_5s_2^2s_1^{-2}s_2^{-4}}+\underline{\underline{
					R_5s_1^{-1}(s_1s_2^2s_1^{-1})(s_1^{-1}s_2^{-2}s_1)s_2^{-2}}}+R_5s_2^2s_1^{-2}s_2^{-1}(s_2^{-1}s_1^{-1}s_2^{-2})+\\&&+R_5s_2(s_2s_1^{-2}s_2^{-1})\grv^{-1}s_2^{-1}+
			\grF(s_2^{-2}s_1^2s_2^2s_1^3s_2^2)\\

			&\stackrel{\ref{ll}}{\in}&R_5s_2^2s_1^{-2}\grv^{-1}s_1^{-1}+
			R_5s_2s_1^{-1}s_2^{-2}s_1\grv^{-1}s_2^{-1}+\grF(U)+U\\

			&\stackrel{\ref{r1}}{\subset}&\underline{s_2^2\grv^{-1}u_1}+
			R_5s_2s_1^{-1}s_2^{-2}\grv^{-1}s_1s_2^{-1}+U\\
			&\subset&s_2s_1^{-1}u_2u_1s_2^{-1}s_1s_2^{-1}+U.
			\end{array}}$
		
		Therefore, by proposition \ref{xx}(ii) we have that the element $s_2^{2}s_1^{-2}s_2^{-2}s_1^{2}s_2^{-2}$ is inside $U$ and, hence, $s_2
		u_1s_2^{-1}s_1^2s_2^{-2}s_1^2s_2^{-2}u_1\subset U$.
		
		In order to finish the proof that $H_5=U$ it remains to prove that $s_2u_1s_2s_1^{-2}s_2^2s_1^{-2}s_2^2\subset U$. For this purpose we expand $u_1$ as 
		$R_5+R_5s_1+R_5s_1^2+R_5s_1^3+
		R_5s_1^4$ and we have:
		
		$\small{\begin{array}[t]{lcl}
			s_2u_1s_2s_1^{-2}s_2^2s_1^{-2}s_2^2
			&\subset&\grF(s_2^{-2}u_1s_2^{-2}s_1^2s_2^{-2})+\underline{\underline{
					R_5(s_2s_1s_2)s_1^{-2}s_2^2s_1^{-2}s_2^2}}+R_5\grv s_1^{-2}s_2^2s_1^{-2}s_2^2+\\&&
			+R_5s_2s_1^3s_2s_1^{-2}s_2^2s_1^{-2}s_2^{2}	+		R_5s_2s_1^4s_2s_1^{-2}s_2^2s_1^{-2}s_2^2.
			\end{array}}$
		
		However, by lemma \ref{lol}(iii) we have that $\grF(s_2^{-2}u_1s_2^{-2}s_1^2s_2^{-2})\subset \grF(U)\stackrel{\ref{r1}}{\subset} U$. Moreover, $\grv s_1^{-2}s_2^2s_1^{-2}s_2^2=\underline{\underline{R_5s_1^{-2}\grv s_2^2s_1^{-2}s_2^2}}$. It remains to prove that $C:=R_5s_2s_1^3s_2s_1^{-2}s_2^2s_1^{-2}s_2^{2}	+		R_5s_2s_1^4s_2s_1^{-2}s_2^2s_1^{-2}s_2^2$ is a subset of $U$. We have:

		$\small{\begin{array}[t]{lcl}
			C
			&=&
			R_5s_2s_1^2(s_1s_2s_1^{-1})s_1^{-1}s_2(s_2s_1^{-2}s_2^{-1})s_2^3+
			R_5s_1^{-1}(s_1s_2s_1^4)s_2s_1^{-2}s_2(s_2s_1^{-2}s_2^{-1})s_2^3\\
			&=&
			R_5(s_2s_1^2s_2^{-1})(s_1s_2s_1^{-1})s_2s_1^{-1}s_2^{-2}s_1s_2^3+R_5s_1^{-1}s_2^4(s_1s_2s_1^{-1})s_1^{-1}(s_2s_1^{-1}s_2^{-1})s_2^{-1}s_1s_2^3\\
			&=&R_5s_1^{-1}s_2(s_2s_1s_2^{-1})(s_1s_2^2s_1^{-1})s_2^{-2}s_1s_2^3+R_5s_1^{-1}s_2^2\grv s_1^{-2}s_2^{-1}s_1s_2^{-1}s_1s_2^3\\
			&=&R_5s_1^{-1}s_2s_1^{-1}(s_2s_1s_2^{-1})s_1^2s_2^{-1}s_1s_2^3+R_5s_1^{-1}s_2^2s_1^{-2}(s_2s_1^3s_2^{-1})s_1s_2^3+U\\
			&=&\underline{\underline{R_5s_1^{-1}s_2s_1^{-2}(s_2s_1^3s_2^{-1})s_1s_2^3}}+u_1s_2^2s_1^{-3}s_2^3s_1^2s_2^3+U. 
		
			\end{array}}$
		
		We expand $s_1^{-3}$ as a linear combination of $s_1^{-2}$, $s_1^{-1}$, 1, $s_1$, $s_1^2$ and $s_1^3$ and we have that 
		
		$\small{\begin{array}{lcl}u_1s_2^2s_1^{-3}s_2^3s_1^2s_2^3
		&\subset& u_1s_2^2s_1^{-2}s_2^3s_1^2s_2^3+\underline{\underline{u_1(s_1s_2^2s_1^{-1})s_2^3s_1^2s_2^3}}+\underline{u_1s_2^5s_1^2s_2^3}+
		\underline{\underline{u_1s_2(s_2s_1s_2^3)s_1^2s_2^3}}+\\&&+
	u_1s_2^2s_1^2s_2^3s_1^2s_2^3
		.
		\end{array}}$
		
		Hence, in order to finish the proof that $H_5=U$ we have to prove that $u_1s_2^2s_1^{-2}s_2^3s_1^2s_2^3$ and $u_1s_2^2s_1^2s_2^3s_1^2s_2^3$ are subsets of $U$. We have:

		$\small{\begin{array}{lcl}
			u_1s_2^2s_1^{-2}s_2^3s_1^2s_2^3&=&u_1s_2^2s_1^{-2}(as_2^2+bs_2+c+ds_2^{-1}+es_2^{-2})s_1^2s_2^3\\
			&\subset&u_1s_2^2s_1^{-2}s_2^2s_1^2s_2^3+
			u_1s_2^2s_1^{-2}s_2^2s_1^2s_2^3+\underline{u_1s_2^5}+\underline{\underline{
					u_1s_2(s_2s_1^{-2}s_2^{-1})s_1^2s_2^3}}+u_1s_2^2s_1^{-2}s_2^{-2}s_1^2s_2^3
			\end{array}}$
		
		However, we have that	$u_1s_2^2s_1^{-2}s_2^2s_1^2s_2^3=
		\underline{\underline{u_1(s_1^{-1}s_2^2s_1)s_2(s_2^{-1}s_1^{-3}s_2)s_1^2s_2^3}}$. Moreover,  we have
		$u_1s_2^2s_1^{-2}s_2^{-2}s_1^2s_2^3=u_1(s_1s_2^2s_1^{-1})(s_1^{-1}s_2^{-2}s_1)s_1s_2^3=	u_1s_2^{-1}s_1^2s_2^3\grv^{-1}s_1s_2^3=u_1s_2^{-1}s_1^2s_2^{3}s_1(s_2^{-1}s_1^{-2}s_2)s_2=u_1s_2^{-1}s_1^2s_2^3s_1^2s_2^{-1}(s_2^{-1}s_1^{-1}s_2)\subset\grF(s_2u_2u_1s_2s_1^{-1}s_2)$. By proposition \ref{xx}(i) and lemma \ref{r1} we have that $u_1s_2^2s_1^{-2}s_2^{-2}s_1^2s_2^3\subset U$. It remains to prove that $u_1s_2^2s_1^{-2}s_2^2s_1^2s_2^3\subset U$. We notice that $u_1s_2^2s_1^{-2}s_2^2s_1^2s_2^3=
		u_1s_2^3(s_2^{-1}s_1^{-2}s_2)s_2s_1^2s_2^3
		=u_1(s_1^{-1}s_2^3s_1)s_2^{-2}s_1^{-1}s_2s_1^2s_2^3
		=u_1s_2s_1^3s_2^{-3}s_1^{-1}s_2s_1^2s_2^3$. We expand $s_2^3$ as a linear combination of $s_2^2$, $s_2$, 1, $s_2^{-1}$ and $s_2^{-2}$ and we have:
		
		$\small{\begin{array}{lcl}
		
		u_1s_2s_1^3s_2^{-3}s_1^{-1}s_2s_1^2s_2^3

		&\subset&u_1s_2s_1^3s_2^{-2}(s_2^{-1}s_1^{-1}s_2)s_1^2s_2^2+u_1s_2s_1^3s_2^{-3}s_1^{-1}\grv+
		\underline{\underline{u_1s_2s_1^3s_2^{-3}s_1^{-1}s_2s_1^2}}+\\&&+
		\underline{\underline{u_1s_2s_1^3s_2^{-3}s_1^{-1}(s_2s_1^2s_2^{-1})}}+u_1s_2s_1^3s_2^{-3}s_1^{-1}(s_2s_1^2s_2^{-1})s_2^{-1}\\
		&\subset&u_1s_2s_1^3s_2^{-2}s_1(s_2^{-1}s_1s_2)s_2+
		u_1s_2s_1^3s_2^{-3}s_1^{-2}s_2(s_2s_1s_2^{-1})+\underline{\underline{u_1s_2s_1^3s_2^{-3}\grv s_1^{-1}}}+U\\
		&\subset&
		u_1(s_2u_1u_2u_1s_2s_1^{-1}s_2)u_1+U.
		\end{array}}$\\
		The result follows from \ref{xx}(i) and \ref{r1}.
		
		Using analogous calculations we will prove that $u_1s_2^2s_1^2s_2^3s_1^2s_2^3$ is a subset of $U$. We have:
		
		$\small{\begin{array}{lcl}
			u_1s_2^2s_1^2s_2^3s_1^2s_2^3&=&u_1s_2^2s_1^2(as_2^2+bs_2+c+ds_2^{-1}+es_2^{-2})s_1^2s_2^3\\
			&\subset&u_1s_2^2s_1^2s_2^2s_1^2s_2^3+u_1s_2\grv s_1^2s_2^3+\underline{u_1s_2^2s_1^4s_2^3}
			+\underline{\underline{u_1s_2(s_2s_1^2s_2^{-1})s_1^2s_2^3}}+
			u_1s_2^2s_1^2s_2^{-2}s_1^2s_2^3.
			\end{array}}$
		
		However, $u_1s_2\grv s_1^2s_2^3=\underline{\underline{u_1s_2s_1^2\grv
				s_2^3}}$. Therefore, it remains to prove that $u_1s_2^2s_1^2s_2^2s_1^2s_2^3$ and 	$u_1s_2^2s_1^2s_2^{-2}s_1^2s_2^3$ are subsets of $U$. We have:
		
		$\small{\begin{array}{lcl}
			u_1s_2^2s_1^2s_2^2s_1^2s_2^3&=&u_1s_2^2s_1^2s_2^2s_1^2(as_2^2+bs_2+c+ds_2^{-1}+es_2^{-2})\\
			&\subset&u_1s_2\grv^2s_2+u_1s_2\grv^2+\underline{u_1s_2^2s_1^2s_2^2s_1^2}+
			u_1s_2\grv s_2s_1^2s_2^{-1}+u_1s_2\grv s_2s_1^2s_2^{-2}\\
			\end{array}}$
		
		By lemma \ref{oo}(ii) we have that $u_1s_2\grv^2s_2+u_1s_2\grv^2\subset u_1s_2s_1^4s_2s_1^3s_2u_1s_2+U$. However, by lemma \ref{oo}(i) we have  $u_1s_2s_1^4(s_2s_1^3s_2u_1s_2)\subset u_1s_2s_1^4(\grv^2u_1+u_1u_2u_1u_2u_1)\subset u_1s_2\grv^2u_1+\underline{\underline{u_1s_2u_1u_2u_1u_2u_1}}$. Using another time lemma \ref{oo}(ii) we have that $u_1s_2s_1^4(\grv^2u_1+u_1u_2u_1u_2u_1)\subset U$. It remains to prove that $D:=u_1s_2\grv s_2s_1^2s_2^{-1}+u_1s_2\grv s_2s_1^2s_2^{-2}$ is a subset of  $U$. We have:

			$\small{\begin{array}{lcl}
			D &=& u_1s_2\grv(s_2s_1^2s_2^{-1})+u_1s_2\grv(s_2s_1^2s_2^{-1})s_2^{-1}\\
			&=&u_1s_2\grv s_1^{-1}s_2^2s_1+u_1s_2\grv s_1^{-1}s_2^2s_1s_2^{-1}\\
			&=&\underline{\underline{u_1s_2s_1^{-1}\grv s_2^2s_1}}+u_1s_2s_1^{-1}\grv s_2^2s_1s_2^{-1}.
			
			\end{array}}$
		
		However, we have $u_1s_2s_1^{-1}\grv s_2^2s_1s_2^{-1}=u_1s_2s_1^{-1}\grv s_2(s_2s_1s_2^{-1})=u_1s_2s_1^{-1}s_2s_1(s_1s_2s_1^{-1})s_2s_1=\underline{\underline{u_1s_2s_1^{-1}(s_2s_1s_2^{-1})s_1s_2^2s_1}},$ meaning that $D \subset U$.
		
		Using analogous calculations we will prove that $u_1s_2^2s_1^2s_2^{-2}s_1^2s_2^3\subset U$. We have:

		$\small{\begin{array}{lcl}
		u_1s_2^2s_1^2s_2^{-2}s_1^2s_2^3
			&=&
			u_1s_2(s_2s_1^2s_2^{-1})s_2^{-1}s_1^2s_2^3\\
		
			&=&u_1s_2s_1^{-1}s_2^2s_1s_2^{-1}s_1^2(as_2^2+bs_2+c+ds_2^{-1}+es_2^{-2})\\
			&\subset&u_1(s_1s_2s_1^{-1})s_2(s_2s_1s_2^{-1})s_1^2s_2^2+
			\underline{\underline{u_1s_2s_1^{-1}s_2^2s_1(s_2^{-1}s_1^2s_2)}}+
			
			\underline{
				\underline{u_1s_2s_1^{-1}s_2^2s_1s_2^{-1}s_1^2}}+\\&&+
			u_1s_2s_1^{-1}s_2(s_2s_1s_2^{-1})s_1^2s_2^{-1}+u_1s_2s_1^{-1}s_2(s_2s_1s_2^{-1})s_1^2s_2^{-2}\\
			
			&\subset&u_1s_2^{-1}(s_1s_2^2s_1^{-1})s_2s_1^3s_2^2
			+\underline{\underline{u_1s_2s_1^{-1}s_2s_1^{-1}(s_2s_1^3s_2^{-1})}}+\\&&+u_1s_2s_1^{-2}(s_1s_2s_1^{-1})(s_2s_1^3s_2^{-1})s_2^{-1}+U\\ 
			&\subset&u_1s_2^{-2}s_1^2s_2^2s_1^3s_2^2+u_1s_2s_1^{-2}s_2^{-1}(s_1s_2s_1^{-1})s_2^2(s_2s_1s_2^{-1})+U\\
			&\stackrel{
				 \ref{ll}}{\subset}&\underline{\underline{u_1s_2s_1^{-2}s_2^{-2}(s_1s_2^3s_1^{-1})s_2s_1}}+U.
			\end{array}}$
		
		\qedhere
	\end{itemize}
\end{proof}
\begin{cor}$H_5$ is a free $R_5$-module of rank $r_5=600$.
\end{cor}
\begin{proof}
	By proposition \ref{rp} it will be sufficient to show that $H_5$ is generated as $R_5$-module by $r_5$ elements. By theorem \ref{thh2} the definition of $U''''$ and the fact that $u_1u_2u_1=u_1+u_1s_2u_1+u_1s_2^{-1}u_1+u_1s_2^2u_1+u_1s_2^{-2}u_1$ we have that $H_5$ is spanned as left $u_1$-module by 120 elements. Since $u_1$ is spanned by 5 elements as a $R_5$-module, we have that $H_5$ is spanned over $R$ by $r_5=600$ elements. 
\end{proof}
\section{The irreducible representations of $B_3$ of dimension at most 5}
\indent

We set $\tilde{R_k}=\ZZ[u_1^{\pm1},...,u_k^{\pm1}]$, $k=2,3,4,5$. Let $\tilde{H_k}$ denote the quotient of the group algebra $\tilde{R_k}B_3$ by the relations $(s_i-u_1)...(s_i-u_k)$, $i=1,2$. In the previous sections we proved that $H_k$ is a free $R_k$-module of rank $r_k$. Hence, $\tilde{H_k}$ is a free $\tilde{R_k}$-module of rank $r_k$  (Lemma 2.3 in \cite{ivan}). We now assume that $\tilde{H_k}$ has a unique symmetrizing trace $t_k: \tilde{H_k} \arw \tilde{R_k}$ (i.e. a trace function such that the bilinear form $(h, h')\mapsto t_k(hh')$ is non-degenerate), having nice properties (see \cite{BMM}, theorem 2.1): for
example, $t_k(1)=1$, which means that $t_k$ specializes to the canonical symmetrizing form on $\CC W_k$. 

Let $\grm_{\infty}$  be the group of all roots of unity in $\CC$. We recall that $W_k$ is the finite quotient group $B_3/\langle s_i^k \rangle$, $k=2, 3, 4,5$ and $i=1,2$. We denote by $K_k$  the \emph{field of definition} of $W_k$, i.e. the number field contained in $\mathbb{Q}(\grm_{\infty})$, which is generated by the traces of all elements of $W_k$ (for more details see \cite{B}).
We denote by $\grm(K_k)$ the group of all roots of unity of $K_k$ and, for every integer $m>1$, we set $\grz_m:=$exp$(2
\grp i/m)$, where $i$  denotes here a square root of -1.

Let $\mathbf{v}=(v_1,...,v_k)$ be a set of $k$ indeterminates such that, for every $i\in\{1,...,k\}$, we have $v_i^{|\grm(K_k)|}=\grz_k^{-i}u_i$. By extension of scalars we obtain a $\CC(\mathbf{v})$-algebra
$\CC(\mathbf{v})\tilde{H_k}:=\tilde{H_k}\otimes_{\tilde{R_k}}\CC(\mathbf{v})$, which is split semisimple (see \cite{Malle}, theorem 5.2). Since the algebra $\CC(\mathbf{v})\tilde{H_k}$ is split, by Tits' deformation theorem (see theorem 7.4.6 in \cite{Geck}), the specialization $v_i \mapsto 1$ induces a bijection
Irr$(\CC(\mathbf{v})\tilde{H_k})\rightarrow$ Irr$(W_k)$.

Let $\varrho: B_3 \arw GL_n(\CC)$ be an irreducible representation of $B_3$ of dimension $k\leq 5$. We set $A:=\varrho(s_1)$ and $B:=\varrho(s_2)$. The matrices $A$ and $B$ are similar since $s_1$ and $s_2$ are conjugate $( s_2=(s_1s_2)s_1(s_1s_2)^{-1})$. Hence, by  Cayley-Hamilton theorem of linear algebra, there is a monic polynomial $m(X)=X^k+m_{n-1}X^{n-1}+...+m_1X+m_0\in \CC[X]$ of degree $k$ such that $m(A)=m(B)=0$. Let $R^k_K$ denotes  the integral closure of
$R_k$ in $K_k$. We fix $\gru: R^k_K\arw \mathbb{C}$ a \emph{specialization} of $R^k_K$, defined by $u_i\mapsto \grl_i$, where $\grl_i$ are the eigenvalues of $A$ (and $B$). We notice that $\gru$ is well-defined, since $m_0=$det$A\in \CC^{\times}$. Therefore, in order to determine $\varrho$ it will be sufficient to  describe the irreducible $\CC\tilde{ H_k}:=\tilde{H_k}\otimes_{\gru}\CC$-modules of dimension $k$.

When the algebra $\CC\tilde H_k$ is semisimple, we can use again Tits' deformation theorem and we have a canonical bijection between  the set of irreducible characters of $\CC\tilde{H_k}$ and  the set of irreducible characters of $\CC(\mathbf{v})\tilde{H_k}$,  which are in bijection with
the irreducible characters of $W_k$. However, this is not always the case. In order to determine the irreducible representations of  $\CC\tilde H_k$ in the general case (when we don't know a priori that  $\CC\tilde H_k$ is semisimple) we use a different approach.

Let $R_0^{+}\big(\CC(\mathbf{v})\tilde{H_k}\big)$ (respectively $R_0^{+}(\CC\tilde{ H_k})$) denote the subset of the \emph{Grothendieck group} of the category of finite dimensional $\CC(\mathbf{v})\tilde{H_k}$ (respectively $\CC\tilde{H_k}$)-modules consisting of elements $[V]$, where $V$ is a $\CC(\mathbf{v})\tilde{H_k}$ (respectively $\CC\tilde{H_k}$)-module (for more details, one may refer to \textsection 7.3 in \cite{Geck}). By theorem 7.4.3 in \cite {Geck} we obtain a well-defined decomposition map $$d_{\gru}: R_0^{+}\big(\CC(\mathbf{v})\tilde{H_k}) \arw R_0^{+}(\CC\tilde{ H_k}).$$
The corresponding \emph{decomposition matrix} is the Irr$\big(\CC(\mathbf{v})\tilde{H_k}\big)\times$ Irr$(\CC\tilde{ H_k})$ matrix $(d_{\grx\grf})$ with non-negative integer entries such that $d_{\gru}([V_{\grx}])=\sum\limits_{\grf}d_{\grx\grf}[V'_{\grf}]$, where $V_{\grx}$ is an irreducible $\CC(\mathbf{v})\tilde{H_k}$-module with character $\grx$ and $V_{\grf}$ is an irreducible $\CC \tilde{H_k}$-module with character $\grf$. 
This matrix records in which way the irreducible representations of
the semisimple algebra $\CC(\mathbf{v})\tilde{H_k}$ break up into irreducible representations 
of $\CC\tilde{ H_k}$. 

The form of the decomposition matrix is controlled by the \emph{Schur elements},  denoted as $s_{\grx}$, $\grx \in$ Irr$\big(\CC(\mathbf{v})\tilde{H_k}\big)$, with respect to the symmetric form $t_k$. The Schur elements belong to $R^k_K$ (see \cite{Geck}, Proposition
7.3.9) and they depend only on the symmetrizing form $t_k$  and the isomorphism class of the representation. Moreover, M. Chlouveraki has shown that these elements are products of cyclotomic polynomials over $K_k$ evaluated on monomials of degree 0 (see theorem 4.2.5 in \cite{Ma}). In the following section we are going to use these elements in order to determine the irreducible representations of  $\CC\tilde H_k$ (for more details about the definition and the properties of the Schur elements, one may refer to \S 7.2 in \cite{Geck}).

We say that the $\CC(\mathbf{v})\tilde{H_k}$-modules $V_{\grx}, V_{\grc}$ \emph{belong to the same block} if the corresponding characters $\grx, \grc$  label the rows of the same block in the decomposition matrix $(d_{\grx\grf})$ (by definition, this means that there is a $\grf\in\text{Irr}(\CC\tilde H_k)$ such that $d_{\grx,\grf}\not=0\not=d_{\grc,\grf}$).
If an irreducible $\CC(\mathbf{v})\tilde{H_k}$-module  is alone in its block, then we call it a \emph{module of defect 0}. Motivated by the idea of M. Chlouveraki and H. Miyachy in \cite{Maria} \textsection 3.1 we use the following criteria in order to determine whether two modules belong to the same block:

\begin{itemize}
	\item We have $\gru(s_{\grx})\not =0$ if and only if $V_{\grx}$ is a module of defect 0 (see \cite{maller}, Lemma 2.6). 
	
	This criterium  together with theorem 7.5.11 in \cite{Geck}, states that $V_{\grx}$ is a module of defect 0 if and only if the decomposition matrix is of the form
	$$\small{\begin{pmatrix} *& \dots&*&0 &*&\dots&*\\
	\vdots& \dots& \vdots&\vdots &\vdots& \dots&\vdots\\
	*& \dots&*&0 &*&\dots&*\\
	0& \dots& 0&1&0&\dots&0\\
	*&  \dots&*&0 &*&\dots&*\\
	\vdots&  \dots&\vdots &\vdots& \vdots& \dots&\vdots\\
	*& \dots&*&0 &*&\dots&*
	\end{pmatrix}}$$
	
	\item If $V_{\grx}, V_{\grc}$ are in  the same block, then $\gru(\grv_{\grx}(z_0))=\gru(\grv_{\grc}(z_0))$ (see \cite{Geck}, Lemma 7.5.10), where $\grv_{\grx}, \grv_{\grc}$ are the corresponding \emph{central characters}\footnote{If $z$  lies in the center of $\CC(\mathbf{v})\tilde{H_k}$ then Schur's lemma implies that $z$ acts as  scalars in $V_{\grx}$ and $V_{\grc}$. We denote these scalars as $\grv_{\grx}(z)$ and $\grv_{\grc}(z)$ and we call the associated $\CC(\mathbf{v})$-homomorphisms $\grv_{\grx},\grv_{\grc}: Z\big(\CC(\mathbf{v})\tilde{H_k}\big)\rightarrow \CC(\mathbf{v})$ central characters (for more details, see \cite{Geck} page 227).} and $z_0$ is the central element $(s_1s_2)^3$.
\end{itemize}

We recall that in order to describe the irreducible representations of $B_3$ of dimension at most 5, it is enough to describe the irreducible $\CC\tilde{H_k}$-modules of dimension $k$. Let $S$ be an irreducible $\CC \tilde{H_k}$-module of dimension $k$ and $s\in S$ with $s\not=0$. The morphism $f_{s}: \CC \tilde{H_k}\arw S$ defined by $h\mapsto hs$ is surjective since $S$ is irreducible. Hence, by the definition of the Grothendieck group we have that $d_{\gru}\big([\CC(\mathbf{v})\tilde{H_k}]\big)=[\CC \tilde{H_k}]=[$kerf$_{s}]+[S]$. However, since $\CC(\mathbf{v})\tilde{H_k}$ is semisimple we have  $\CC(\mathbf{v})H_k=M_1\oplus...\oplus M_r$, where the $M_i$ are (up to isomorphism) all the simple $\CC(\mathbf{v})\tilde{H_k}$-modules (with redundancies). Therefore, we have $\sum\limits_{i=1}^{r}d_{\gru}([M_i])=[$ker$f_{s}]+[S].$ 
Hence, there is a simple  $\CC(\mathbf{v})\tilde{H_k}$-module $M$ such that 
\begin{equation}d_{\gru}([M])=[S]+[J],\label{eqqq}\end{equation} where $J$ is a $\CC \tilde{H_k}$-module.
\begin{rem}
	\mbox{}  
	\vspace*{-\parsep}
	\vspace*{-\baselineskip}\\ 
	\begin{itemize}
		\item[(i)] The $\CC(\mathbf{v})\tilde{H_k}$-module $M$ is of dimension at least $k$.
		\item[(ii)] If $J$ is of dimension 1, there is a $\CC(\mathbf{v})\tilde{H_k}$-module $N$ of dimension 1, such that $d_{\gru}([N])=[J]$. This result comes from the fact that the 1-dimensional $\CC \tilde H_k$-modules are of the form $(\grl_i)$, $i=1,\dots,k$ and, by definition, $\grl_i=\gru(u_i)$.
	\end{itemize}	
	\label{brrrr}
\end{rem}
The irreducible  $\CC(\mathbf{v})\tilde{H_k}$-modules are known (see \cite{Mallem} or \cite{BM} \textsection 5B and \textsection 5D, for $n=3$ and $n=4$, respectively). Therefore, we
can determine $S$ by using (\ref{eqqq}) and a case-by-case analysis. 
\begin{itemize}[leftmargin=*]
	
	\item \underline{$k=2$} : Since $\tilde{H_2}$ is the generic Hecke algebra of $\mathfrak{S}_3$, which is a Coxeter group, the irreducible representations of $\CC \tilde{H_2}$ are well-known; we have two irreducible representations of dimension 1 and one of dimension 2. By $(\ref{eqqq})$ and remark \ref{brrrr} (i),  $M$ must be the irreducible $\CC(\mathbf{v})\tilde{H_k}$-module of dimension 2 and $(\ref{eqqq})$ becomes $[S]=d_{\gru}([M])$. Hence, we have: 
	$$A=\begin{bmatrix}
	\begin{array}{rr}
	\grl_1&\grl_1\\
	0&\grl_2
	\end{array}
	\end{bmatrix},\;
	B=\begin{bmatrix}
	\begin{array}{rr}
	\grl_2&0\\
	-\grl_2&\grl_1\\
	\end{array}
	\end{bmatrix}$$
	Moreover, $[S]=d_{\gru}([M])$ is irreducible and $M$ is the only irreducible $\CC(\mathbf{v})\tilde{H_k}$-module of dimension 2. As a result, $M$ has to be alone in its block i.e. $\gru(s_{\grx})\not=0$, where $\grx$ is the character that corresponds to $M$. Therefore, an irreducible representation of $B_3$ of dimension 2 can be described by the explicit matrices $A$ and $B$ we have above, depending only on a choice of $\grl_1, \grl_2$ such that $\gru(s_{\grx})=\grl_1^2-\grl_1\grl_2+\grl_2^2\not=0$.
	
	\item \underline{$k=3$} : Since the algebra $\CC(\mathbf{v})\tilde{H_3}$ is split, we have a bijection between the set 
	Irr$(\CC(\mathbf{v})\tilde{H_3})$ and the set Irr$(W)$, as we explained earlier.
	We refer to J. Michel's version of CHEVIE package of GAP3 (see \cite{J}) in order to find the irreducible characters of $W_3$. 
	We type:
	\begin{verbatim}
	gap> W_3:=ComplexReflectionGroup(4);
	gap> CharNames(W_3);
	[ "phi{1,0}", "phi{1,4}", "phi{1,8}", "phi{2,5}", "phi{2,3}", "phi{2,1}",
	"phi{3,2}" ]
	\end{verbatim}
	We have 7 irreducible characters $\grf_{i,j}$, where $i$ is the dimension of the representation and $j$ the valuation of its fake degree (see \cite{Malle} \textsection 6A). Since $S$ is of dimension 3,  the equation $(\ref{eqqq})$ becomes $[S]=d_{\gru}([M])$, where $M$ is the irreducible $\CC(\mathbf{v})\tilde{H_3}$-module  that corresponds to the character $\grf_{3,2}$ (see remark \ref{brrrr}(i)). However, we have explicit matrix models for this representation (see \cite{BM}, \textsection 5B or we can refer to CHEVIE package of GAP3 again) and since $[S]=d_{\gru}([M])$ we have:
	$$A=\begin{bmatrix}
	\grl_3&0&0\\
	\grl_1\grl_3+\grl_2^2& \grl_2& 0\\
	\grl_2& 1&\grl_1
	\end{bmatrix},\;
	B=\begin{bmatrix}
	\grl_1& -1&\grl_2\\
	0&\grl_2&-\grl_1\grl_3-\grl_2^2\\
	0&0&\grl_3
	\end{bmatrix}.$$
	$M$ is the only  irreducible $\CC(\mathbf{v})\tilde{H_3}$-module of dimension 3, therefore, as in the case where $k=2$, we must have that $\gru(s_{\grf_{3,2}})\not=0$. The Schur element $s_{\grf_{3,2}}$ has been determined in \cite{MAA} and the condition $\gru(s_{\grf_{3,2}})\not=0$ becomes \begin{equation}\frac{(\grl_1^2+\grl_2\grl_3)(\grl_2^2+\grl_1\grl_3)(\grl_3^2+\grl_1\grl_2)}{(\grl_1\grl_2\grl_3)^2}\not=0.\label{tt1}\end{equation}
	To sum up, an irreducible representation of $B_3$ of dimension 3 can be described by the explicit matrices $A$ and $B$ we gave above, depending only on a choice of $\grl_1, \grl_2, \grl_3$ such that (\ref{tt1}) is satisfied. 
	
	\item \underline{$k=4$} : We use again the program GAP3 package CHEVIE in order to find the irreducible characters  of $W_4$. In this case we have 16 irreducible characters among which 2 of dimension 4; the characters $\grf_{4,5}$ and $\grf_{4,3}$ (we follow again the notations in GAP3, as in the case where $k=3$). 
	
	By remark \ref{brrrr}(i) and relation $(\ref{eqqq})$, we have $[S]=d_{\gru}([M])$, where $M$ is the irreducible $\CC(\mathbf{v})\tilde{H_4}$-module  that corresponds either to the character $\grf_{4,5}$ or to the character $\grf_{4,3}$. We have again explicit matrix models for these representations (see \cite{BM}, \textsection 5D, where we multiply the matrices described there by a scalar $t$ and we set $u_1=t, u_2=tu, u_3=tv$ and $u_4=tw$):
	$$A=\begin{bmatrix}
	\grl_1&0&0&0\\ \\
	\frac{\grl_1^2}{\grl_2}&\grl_2& 0& 0\\\\
	\frac{\grl_1^3}{r}&\frac{\grl_1\grl_2\grl_3-\grl_1r}{r}& \grl_3& 0\\\\
	-\grl_2&\grl_2\gra&\frac{r\gra}{\grl_1^2}&\grl_4
	\end{bmatrix},\;
	B=\begin{bmatrix}
	\grl_4&\grl_3\gra&\frac{\grl_2\grl_3\gra}{\grl_1}&-\frac{\grl_2\grl_3^2}{r}\\\\
	0&\grl_3&\frac{\grl_2\grl_3-r}{\grl_1}&\frac{\grl_1^2\grl_3}{r}\\\\
	0&0&\grl_2&\frac{\grl_1^3}{r}\\\\
	0&0&0&\grl_1
	\end{bmatrix},$$
	where $r:=\pm\sqrt{\grl_1\grl_2\grl_3\grl_4}$ and $\gra:=\frac{r-\grl_2\grl_3-\grl_1\grl_4}{\grl_1^2}$.
	
	Since $d_{\gru}([M])$ is irreducible either $M$ is of defect 0 or it is in the same block as the other irreducible module of dimension 4  i.e. $\gru(\grv_{\grf_{4,5}}(z_0))=\gru(\grv_{\grf_{4,3}}(z_0))$. We use the program GAP3 package CHEVIE in order to calculate  these central characters. 
	
	More precisely, we have 16 representations where the last 2 are of dimension 4. These representations will be noted in GAP3 as $\verb+R[15]+$ and $\verb+R[16]+$. Since $z_0=(s_1s_2)^3$  we need to calculate the matrices $R[i](s_1s_2s_1s_2s_1s_2), i=15, 16$. These are the matrices $\verb+Product(R[15]{[1,2,1,2,1,2]})+$ and $\verb+Product(R[16]{[1,2,1,2,1,2]})+$,
	 in GAP3 notation, as we can see below:
	\begin{verbatim}
	gap> R:=Representations(H_4);;
	gap> Product(R[15]{[1,2,1,2,1,2]});
	[ [ u_1^3/2u_2^3/2u_3^3/2u_4^3/2, 0, 0, 0 ], 
	[ 0, u_1^3/2u_2^3/2u_3^3/2u_4^3/2, 0, 0 ],
	[ 0, 0, u_1^3/2u_2^3/2u_3^3/2u_4^3/2, 0 ],
	[ 0, 0, 0, u_1^3/2u_2^3/2u_3^3/2u_4^3/2] ]
	gap> Product(R[16]{[1,2,1,2,1,2]});
	[ [ -u_1^3/2u_2^3/2u_3^3/2u_4^3/2, 0, 0, 0 ], 
	[ 0, -u_1^3/2u_2^3/2u_3^3/2u_4^3/2, 0, 0 ],
	[ 0, 0, -u_1^3/2u_2^3/2u_3^3/2u_4^3/2, 0 ],
	[ 0, 0, 0, -u_1^3/2u_2^3/2u_3^3/2u_4^3/2 ] ]
	\end{verbatim}
	We have $\gru(\grv_{4,5}(z_0))=-\gru(\grv_{4,3}(z_0))$, which means that 
	$M$ is of defect zero i.e. $\gru(s_{\grf_{4,i}})\not=0$, where $i=3$ or 5.  The Schur elements $s_{\grf_{4,i}}$ have been determined in \cite{MAA} \textsection 5.10, hence the condition $\gru(s_{\grf_{4,i}})\not =0$ becomes:
	\begin{equation}\frac{-2r\prod\limits_{p=1}^4(r+\grl_p^2)\prod\limits_{r,l}(r+\grl_r\grl_l+\grl_s\grl_t)}{(\grl_1\grl_2\grl_3\grl_4)^4}\not=0,
	\text {where } \{r,l,s,t\}=\{1,2,3,4\}
	\label{tt2}
	\end{equation}
	Therefore, an irreducible representation of $B_3$ of dimension 4 can be described by the explicit matrices $A$ and $B$ depending only on a choice of $\grl_1, \grl_2, \grl_3, \grl_4$ and a square root of $\grl_1\grl_2\grl_3\grl_4$ such that (\ref{tt2}) is satisfied. 
	\item \underline{$k=5$} : In this case, compared to the previous ones, we have two possibilities for $S$. The reason is that we have characters of dimension 5 and  dimension 6, as well. Therefore, by remark \ref{brrrr}(i) and (ii) and (\ref{eqqq})  we either have $d_{\gru}([M])=[S]$, where $M$ is one  irreducible $\CC(\mathbf{v})\tilde{H_5}$-module of dimension 5 or  $d_{\gru}([N])=[S]+d_{\gru}([N'])$, where $N, N'$ are some irreducible $\CC(\mathbf{v})\tilde{H_5}$-modules of dimension 6 and 1, respectively.
	
	In order to exclude the latter case, it is enough to show that $N$ and $N'$ are not in the same block. Therefore, at this point, we may assume that $\gru(\grv_{\grx}(z_0))\not=\gru(\grv_{\grc}(z_0))$, for every irreducible character $\grx, \grc$ of $W_5$ of dimension 6 and 1, respectively. We use GAP3 in order to calculate the central characters, as we did in the case where $k=4$ and we have:
	$\gru(\grv_{\grc}(z_0))=\grl_i^6$, $i\in \{1,...,5\}$ and  $\gru(\grv_{\grx}(z_0))=-x^2yztw$, where $\{x,y,z,t,w\}=\{\grl_1, \grl_2, \grl_3, \grl_4, \grl_5\}$. We notice that $\gru(\grv_{\grx}(z_0))=-\grl_j$det$A$, $j\in \{1,...,5\}$.  Therefore, the assumption  $\gru(\grv_{\grx}(z_0))\not=\gru(\grv_{\grc}(z_0))$ becomes det$A\not=-\grl_i^6\grl_j^{-1}$, $i,j \in\{1,2,3,4,5\}$, where $i, j$ are not necessarily distinct.
	
	By this assumption we have that $d_{\gru}([M])=[S]$, where $M$ is some irreducible $\CC(\mathbf{v})\tilde{H_5}$-module of dimension 5. We have again explicit matrix models for these representation (see \cite{Mallem} or the CHEVIE package of GAP3), therefore we can determine the matrices $A$ and $B$.
	We notice that these matrices depend only on the choice of eigenvalues and of a fifth root of det$A$.
	
	Since $d_{\gru}([M])$ is irreducible either $M$ is of defect 0 or it is in the same block with another irreducible module of dimension 5. However, since the central characters of the irreducible modules of dimension 5 are distinct fifth roots of  $(u_1u_2u_3u_4u_5)^{6}$, we can exclude the latter case.
	Hence, $M$ is of defect zero i.e. $\gru(s_{\grf})\not=0$, where $\grf$ is the character that corresponds to $M$. The Schur elements  have been determined in \cite{MAA} (see also Appendix A.3 in \cite{Ma}) and one can also find them in CHEVIE package of GAP3; they are 
	$$\frac{5\prod\limits_{i=1}^5(r+u_i)(r-\grz_3u_i)(r-\grz_3^2u_i)\prod\limits_{i\not=j}(r^2+u_iu_j)}{(u_1u_2u_3u_4u_5)^7},$$
	where $r$ is a 5th root of $u_1u_2u_3u_4u_5$. However, due to the assumption det$A\not=-\grl_i^5$, $i \in\{1,2,3,4,5\}$ (case where $i=j$), we have that $\gru(r)+\grl_i \not =0$. Therefore, the condition $\gru(s_{\grf})\not=0$ becomes
	\begin{equation}\prod\limits_{i=1}^5(\tilde{r}^2+\grl_i\tilde{r}+\grl_i^2)\prod\limits_{i\not=j}(\tilde{r}^2+\grl_i\grl_j)\not=0,
	\label{oua}
	\end{equation}
	where $\tilde{r}$ is a fifth root of det$A$. 
	
	To sum up, an irreducible representation of $B_3$ of dimension 5 can be described by the explicit matrices $A$ and $B$, that one can find for example in CHEVIE package of GAP3, depending only on a choice of $\grl_1, \grl_2, \grl_3, \grl_4, \grl_5$ and a fifth root of det$A$ such that (\ref{oua}) is satisfied.
\end{itemize}
\begin{rem}
	\begin{enumerate}
		\mbox{}  
		\vspace*{-\parsep}
		\vspace*{-\baselineskip}\\
		\item We can generalize our results for a representation of $B_3$ over a field of positive characteristic, using similar arguments. However, the cases where $k=4$ and $k=5$ need some extra analysis; For $k=4$ we have two irreducible $\CC(\mathbf{v})\tilde{H_4}$-modules of dimension 4, which are not in the same block if we are in any characteristic but 2. However, when we are in characteristic 2, these two modules coincide and, therefore, we obtain an irreducible module of $B_3$ which is of defect 0, hence we arrive to the same result as in characteristic 0. We have exactly the same argument for the case where $k=5$ and we are over a field of characteristic 5.

		\item The irreducible representations of $B_3$ of dimension at most 5 have been classified in \cite{T}. Using a new framework, we arrived to the same results. The matrices $A$ and $B$ described by Tuba and Wenzl are the same (up to equivalence) with the matrices we provide in this paper. For example, in the case where $k=3$, we have given explicit matrices $A$ and $B$. If we take the matrices  $DAD^{-1}$ and $DBD^{-1}$, where $D$ is the invertible matrix $$D=\begin{bmatrix}
		-\grl_1\grl_2-\grl_3^2&\grl_1(\grl_3-\grl_1)&(\grl_2-\grl_3)(\grl_3-\grl_1)\\
		(\grl_2-\grl_1)(\grl_3^2+\grl_1\grl_2)& \grl_1(2\grl_2\grl_1-\grl_1^2+2\grl_1\grl_3-\grl_3\grl_2)& (\grl_1-\grl_3)(\grl_2^2+\grl_1\grl_3)\\
		0& \grl_1(\grl_1-\grl_3)&-\grl_3\grl_1(\grl_1+\grl_2)
		\end{bmatrix},$$ we just obtain the matrices determined in \cite{T} (The matrix $D$ is invertible since det$D=\grl_1(\grl_1^2+\grl_2\grl_3)(\grl_3^2+\grl_1\grl_2)^2\not=0$, due to (\ref{eqqq})).
	\end{enumerate}
\end{rem}

\end{document}